\documentclass[a4paper]{article}

\usepackage{amsmath,amsthm}
\usepackage{amsfonts}
\usepackage{amscd}
\usepackage[utf8]{inputenc} 
\usepackage{indentfirst} 
\usepackage{array} 
\usepackage{colortbl} 
\usepackage[all]{xy} 
\usepackage{mathrsfs}

\allowdisplaybreaks[1] 

\newenvironment{sumibmatrix}{\left [ \begin{smallmatrix}} {\end{smallmatrix}\right ]}

\makeatletter
    
    \@addtoreset{equation}{section}
  \makeatother

\theoremstyle{definition}
\newtheorem{theo}{\bf{Theorem}}[section]
\newtheorem{defi}[theo]{\bf{Definition}}
\newtheorem{lemm}[theo]{\bf{Lemma}}
\newtheorem{cor}[theo]{\bf{Corollary}}
\newtheorem{prop}[theo]{\bf{Proposition}}
\newtheorem{rem}[theo]{\bf{Remark}}
\newtheorem{ex}[theo]{\bf{Example}}
\newtheorem*{ack*}{Acknowledgements}

\begin{document}
\title{Two FRT bialgebroids and their relations}
\author {Yudai Otsuto\thanks{Department of Mathematics, Faculty of Science, Hokkaido University, Sapporo 0600810, Japan; my.otsuto@math.sci.hokudai.ac.jp}}

\date{}
\maketitle

\begin{abstract}
We generalize the FRT construction for the quiver-theoretical quantum Yang-Baxter equation and obtain a left bialgebroid \( \mathfrak{A}(w) \). There are some relations between the left bialgebroid \( \mathfrak{A}(w) \) and a left bialgebroid \( A_{\sigma} \) by the FRT construction for the dynamical Yang-Baxter map. 
\end{abstract}

\section{Introduction}
The quantum Yang-Baxter equation (QYBE for short) \cite{bax,mc} plays an important role in the study of bialgebras. Coquasitriangular bialgebras give birth to solutions to the QYBE. On the other hand, Faddeev, Reshetikhin, and Takhtajan \cite{frt} introduced construction of coquasitriangular bialgebras using solutions to the QYBE, called the FRT construction. This construction has been generalized with the development of the study of the QYBE and bialgebras.

The quantum dynamical Yang-Baxter equation (QDYBE for short), a generalization of the QYBE, was introduced by Gervais and Neveu \cite{gn}. Dynamical R-matrices, solutions to the QDYBE, give birth to \( \mathfrak{h} \)-bialgebroids introduced by Etingof and Varchenko \cite{etva}. If a dynamical R-matrice satisfies a certain condition, called rigidity, this \( \mathfrak{h} \)-bialgebroid has an antipode and is called an \( \mathfrak{h} \)-Hopf algebroid.

A set-theoretical analogue of the QDYBE is the dynamical Yang-Baxter map (DYBM for short) introduced by Shibukawa \cite{shibudy, shibuin}. The DYBM is a generalization of the Yang-Baxter map \cite{ets, ves} suggested by Drinfel'd \cite{dri}. Shibukawa and Takeuchi studied the FRT construction for the DYBM in \cite{shibufr,shiburi}. A left bialgebroid \( A_{\sigma} \) is obtained by this construction. If a solution \( \sigma \) satisfies rigidity, then \( A_{\sigma} \) becomes a Hopf algebroid with a bijective antipode. The notion of left bialgebroids (Takeuchi's \( \times_R \)-bialgebras) was introduced in \cite{take}. This is a genralization of the bialgebra using a non-commutative base algebra \( R \). Its comultiplication and counit are \( (R, R) \)-bimodule homomorphisms. Schauenburg \cite{schau} proposed a Hopf algebraic structure on the left bialgebroid without an antipode, called a \( \times_R \)-Hopf algebra. As a special case of the \( \times_R \)-Hopf algebra, B\"ohm and Szl\'achanyi \cite{bosz} introduced the Hopf algebroid, which has a bijective antipode. 

On the other hand, Hayashi \cite{haI} introduced the notion of face algebras. In \cite{hayas}, a coquasitriangular face algebra \( \mathfrak{A}(w) \) was constructed using a solution \( w \) to the quiver-theoretical QYBE. In addition, if a  coquasitriangular face algebra satisfies the closurability, this face algebra has a Hopf closure, which is a Hopf face algebra satisfying a certain universal property. Hayashi \cite{hayas} constructed this Hopf closure by using the double cross product and the localization of the face algebra. Later (Hopf) face algebras were integrated to weak bialgebras (weak Hopf algebras) by  B\"ohm, Nill, and Szl\'achanyi \cite{boszn}. Bennoun and Pfeiffer mentioned to the Hopf closure (they called the Hopf envelope) of the coquasitriangular weak bialgebra in \cite{benp}. Schauenburg \cite{schau} showed that a weak bialgebra (weak Hopf algebra) is a left bialgebroid (\( \times_R \)-Hopf algebra) whose base algebra is Frobenius-separable. Conversely, a left bialgebroid (\( \times_R \)-Hopf algebra) becomes a weak bialgebra (weak Hopf algebra) if its base algebra is Frobenius-separable. 

There is an interesting relation between Hayashi's generalization and Shibukawa-Takeuchi's generalization of the FRT construction. If the parameter set \( \Lambda \) of a DYBM \( \sigma \) is finite, the left bialgebroid \( A_{\sigma} \) becomes a weak bialgebra since the base algebra \( M_{\Lambda}(\mathbb{K}) \) consisting of maps from \( \Lambda \) to a field \( \mathbb{K} \) is a Frobenius-separable algebra over \( \mathbb{K} \). Matsumoto and Shimizu \cite{matsu} showed that a DYBM \( \sigma \) gives birth to a solution \( w_{\sigma} \) to the quiver-theoretical YBE and gave a weak bialgebra homomorphism \( \Phi \) from \( \mathfrak{A}(w_{\sigma}) \) to \( A_{\sigma} \). 

Shibukawa and the auther generalized the FRT construction for the DYBM to get a left bialgebroid (Hopf algebroid) \( A_{\sigma} \) that is not a weak bialgebra (weak Hopf algebra) from a DYBM with a finite parameter set \( \Lambda \). In \cite{oshibu}, we generalized \( M_{\Lambda}(\mathbb{K}) \) to an arbitrary \( \mathbb{K} \)-algebra \( L \).

It is natural to try to get a left bialgebroid \( \mathfrak{A}(w_{\sigma}) \) corresponding to the generalized left bialgebroid \( A_{\sigma} \).

The purpose of this paper is to discuss relations of two generalized FRT constructions by extending Matsumoto-Shimizu's homomorphism. Let \( R \) be an arbitrary \( \mathbb{K} \)-algebra and we denote by \( M_{\Lambda}(R) \) the \( \mathbb{K} \)-algebra consisting of maps from \( \Lambda \) to \( R \). We try to generalize Hayashi's construction to gain a left bialgebroid \( \mathfrak{A}(w_{\sigma}) \) corresponding to the generalized \( A_{\sigma} \) with the base algebra \( L = M_{\Lambda}(R) \) and construct a left bialgebroid homomorphism from \( \mathfrak{A}(w_{\sigma}) \) to \( A_{\sigma} \). We also show that the weak Hopf algebra \( A_{\sigma} \) with the Frobenius-separable base algebra \( M_{\Lambda}(R) \) becomes a Hopf closure of \( \mathfrak{A}(w_{\sigma}) \) through \( \Phi \).

The paper is organized as follows. In Section \ref{sec:lw}, we review relations between left bialgebroids (Hopf algebroids) and weak bialgebras (weak Hopf algebras) following \cite{bosz} and \cite{schau}. In Section \ref{sec:TL}, we recall a left bialgebroid (Hopf algebroid) \( A_{\sigma} \) with the base algebra \( M_{\Lambda}( R ) \) in \cite{oshibu} and introduce a left bialgebroid \( \mathfrak{A}(w) \) as a generalization of \cite{hayas}. The weak bialgebra \( \mathfrak{A}(w) \) in \cite{hayas} has the base algebra \( M_{\Lambda}( \mathbb{K} ) \) as a left bialgebroid. We generalize \( M_{\Lambda}( \mathbb{K} ) \) to \( M_{\Lambda}( R ) \) like the above left bialgebroid \( A_{\sigma} \). In Section \ref{sec:Phi}, we induce a left bialgebroid \( \mathfrak{A}(w_{\sigma}) := \mathfrak{A}(w) \) by using the setting of the left bialgebroid \( A_{\sigma} \) and construct a left bialgebroid homomorphism \( \Phi \) from \( \mathfrak{A}(w_{\sigma}) \) to \( A_{\sigma} \). As a point of difference between \cite{matsu} and this paper,  we do not use DYBMs to construct these \( \mathfrak{A}(w_{\sigma}) \) and \( \Phi \). We also give an example of the left bialgebroids \( A_{\sigma} \), \( \mathfrak{A}(w_{\sigma}) \) and \( \Phi \) not using the DYBM.  In Section \ref{sec:pro}, if the \( \mathbb{K} \)-algebra \( R \) is a Frobenius-separable \( \mathbb{K} \)-algebra and \( A_{\sigma} \) is a weak Hopf algebra, \( \mathfrak{A}(w_{\sigma}) \), \( A_{\sigma} \), and \( \Phi \) satisfy a certain universal property, called the Hopf closure. To complete this purpose, we introduce the notion of the antipode of the \( \mathbb{K} \)-algebra homomorphism whose domain is a weak bialgebra. This is a generalization of Hayashi's antipode with respect to the face algebra in \cite{hayas}. We can characterize weak bialgebras and generalize Hopf envelopes in \cite{benp} by using these antipodes.

\section{Preliminaries} \label{sec:lw}

In this section, we recall the notion of left bialgebroids (Hopf algebroids) and discuss relation with weak bialgebras (weak Hopf algebras). If the base algebra of a left bialgebroid is a Frobenius-separable algebra, the total algebra has a weak bialgebra structure. In addition, the total algebra is a weak Hopf algebra when the left bialgebroid becomes a Hopf algebroid. For more details, we refer to \cite{bosz} and \cite{schau}.

Throughout this paper, we denote by \( \mathbb{K} \) a field.

\begin{defi}
Let \( A \) and \( L \) be \( \mathbb{K} \)-algebras. A left bialgebroid (or Takeuchi's \( \times_L \)-bialgebra) \( \mathcal{A}_{L} \) is a sextuplet \(\mathcal{A}_{L} := (A,L,s_L,t_L,\Delta_L,\pi_L)\) satisfying the following conditions: 

\begin{enumerate}
\item
The maps \( s_L \colon L \to A \) and \( t_L \colon L^{op} \to A \) are \( \mathbb{K} \)-algebra homomorphisms and satisfy
\begin{equation}
s_L(l)t_L(l^{\prime}) = t_L(l^{\prime})s_L(l)\;\;(\forall l, l^{\prime} \in L). \label{def:stcom}
\end{equation}
Here \( L^{op} \) means the opposite \( \mathbb{K} \)-algebra of \( L \). These two homomorphisms make \( A \) an \( (L, L) \)-bimodule \( _L A_L \) by the following left and right \( L \)-module structures \( _L A \) and \( A_L \):
\begin{equation}
_L A \colon l \cdot a = s_L(l) a; \;\; A_L \colon a \cdot l = t_L(l) a \;\; ( l \in L, a \in A ). \label{def:ac}
\end{equation}

\item
The triple \( (_L A_L, \Delta_L, \pi_L) \) is a comonoid in the category of \( (L, L) \)-bimodules such that
\begin{align}
&a_{[1]}t_L(l) \otimes a_{[2]} = a_{[1]} \otimes a_{[2]}s_L(l); \label{def:st} \\
&\Delta_L(1_A) = 1_A \otimes 1_A; \label{def:Duni} \\
&\Delta_L(ab) = \Delta_L(a)\Delta_L(b); \label{def:Dmulti} \\
&\pi_L(1_A) = 1_L; \label{def:punit} \\
&\pi_L(as_L(\pi_L(b))) = \pi_L(ab) = \pi(at_L(\pi_L(b))) \label{def:pmulti}
\end{align}
for all \( l \in L \) and \( a, b \in A \). Here we write \(\Delta_L(a) = a_{[1]} \otimes a_{[2]}\), called Sweedler's notation. The right-hand-side of \eqref{def:Dmulti} is well defined because of \eqref{def:st}.
\end{enumerate}
We write \( \mathcal{A}_{L} = (A,L,s^A_L,t^A_L,\Delta^A_L,\pi^A_L) \) if there is a possibility of confusion. For a left bialgebroid \( \mathcal{A}_L \), these \( \mathbb{K} \)-algebras \( A \) and \( L \) are called the total algebra and the base algebra, respectively.
\end{defi}

\begin{defi}
Let \( \mathcal{A}_L = (A, L, s_L, t_L, \Delta_L, \pi_L) \) and \( \mathcal{A}^{\prime}_{L^{\prime}} = (A^{\prime}, L^{\prime}, s_{L^{\prime}}, t_{L^{\prime}}, \Delta_{L^{\prime}}, \pi_{L^{\prime}}) \) be left bialgebroids. A pair of \( \mathbb{K} \)-algebra homomorphisms \( (\Phi \colon A \to A^{\prime}, \phi \colon L \to L^{\prime}) \) is called a left bialgebroid homomorphism \( \mathcal{A}_L \to \mathcal{A}^{\prime}_{L^{\prime}} \), iff
\begin{align}
&s_{L^{\prime}} \circ \phi = \Phi \circ s_L; \label{def:sp} \\
&t_{L^{\prime}} \circ \phi = \Phi \circ t_L;  \label{def:tp} \\
&\pi_{L^{\prime}} \circ \Phi = \phi \circ \pi_L;  \label{def:piP} \\
&\Delta_{L^{\prime}} \circ \Phi = ( \Phi \otimes \Phi ) \circ \Delta_L. \label{def:Dp}
\end{align}
The map \( \Phi \otimes \Phi : A \otimes_L A \to A^{\prime} \otimes_{L^{\prime}} A^{\prime} \) makes sense because of \eqref{def:sp} and \eqref{def:tp}.
\end{defi}

Let \( \mathcal{A}_L := (A, L, s_L, t_L, \Delta_L, \pi_L) \) be a left bialgebroid and \( N \) a \( \mathbb{K} \)-algebra isomorphic to the opposite \( \mathbb{K} \)-algebra \( L^{op} \). 
We suppose that \( A \) has a \( \mathbb{K} \)-algebra anti-automorphism \( S \) satisfying
\begin{align}
S \circ t_L &= s_L; \label{St} \\
S( a_{[1]} ) a_{[2]} &= t_L \circ \pi_L \circ S( a ) \label{Sa}
\end{align}
for all \( a \in A \). The left hand side of \eqref{Sa} makes sense because of \eqref{St}. We fix a \( \mathbb{K} \)-algebra isomorphism \( \omega \colon L^{op} \to N \). Then \( A \) has left and right \( N \)-module structures \( ^NA \) and \( A^N \) through the following actions:
\begin{equation}
^N A \colon r \cdot a = a s_L \circ \omega^{-1}( n ); \;\; A^N \colon a \cdot r = a S \circ s_L \circ \omega^{-1}( n ) \;\; (a \in A, n \in N). \label{NAN}
\end{equation}
By virtue of \eqref{St}, these two actions \eqref{NAN} make \( A \) an \( (N, N) \)-bimodule. We can also define two \( \mathbb{K} \)-linear maps \( S_{A \otimes_{L} A}  \) and  \( S_{A \otimes_{N} A} \) by
\begin{align}
S_{A \otimes_{L} A} \colon A \otimes_L A \ni a \otimes b \mapsto S(b) \otimes S(a) \in A \otimes_N A; \label{SAL} \\
S_{A \otimes_{N} A} \colon A \otimes_N A \ni a \otimes b \mapsto S(b) \otimes S(a) \in A \otimes_L A. \label{SAN}
\end{align}

\begin{defi}
Let \( (\mathcal{A}_L, S) \) be a pair of a left bialgebroid \( \mathcal{A}_L \) and a \( \mathbb{K} \)-algebra anti-automorphism \( S \colon A \to A \) satisfying \eqref{St} and \eqref{Sa}. Suppose that \( S_{A \otimes_{N} A} \) has the inverse \( S_{A \otimes_{N} A}^{-1} \). We say that the pair \( (\mathcal{A}_L, S) \) is a Hopf algebroid, iff
\begin{align}
S_{A \otimes_{L} A} \circ \Delta_L \circ S^{-1} &= S_{A \otimes_{N} A}^{-1} \circ \Delta_L \circ S; \\
(\Delta_L \otimes {\rm id}_A) \circ \Delta_N &= ({\rm id}_A \otimes \Delta_N) \circ \Delta_L; \\
(\Delta_N \otimes {\rm id}_A) \circ \Delta_L &= ({\rm id}_A \otimes \Delta_L) \circ \Delta_N.
\end{align}
Here we define \( \Delta_N = S_{A \otimes_{L} A} \circ \Delta_L \circ S^{-1} \). The map \( S \) is called an antipode.
\end{defi}

We next introduce the notion of weak bialgebras and weak Hopf algebras.

\begin{defi} \label{def:WBA}
Let \( B \) a \( \mathbb{K} \)-algebra endowed with a \( \mathbb{K} \)-coalgebra structure by \( \Delta \colon B \to B \otimes_{\mathbb{K}} B \) and \( \varepsilon \colon B \to \mathbb{K} \). We say that a triple \( B := (B, \Delta, \varepsilon) \) is a weak bialgebra if the folloing conditions are satisfied:
\begin{align}
\Delta(ab) =& \Delta(a)\Delta(b); \label{def:wdm} \\
(\Delta(1) \otimes 1)(1 \otimes \Delta(1)) = 1_{(1)} \otimes 1_{(2)}& \otimes 1_{(3)} = (1 \otimes \Delta(1))(\Delta(1) \otimes 1); \label{def:D13} \\
\varepsilon( a b_{(1)} ) \varepsilon( b_{(2)} c ) = \varepsilon&( abc ) = \varepsilon( a b_{(2)} ) \varepsilon( b_{(1)} c ) \label{def:cum}
\end{align}
for all \( a, b, c \in B \). Here we write simply \( 1 = 1_B \) and use Sweedler's notation, which is written by
\begin{equation*}
\Delta(a) = a_{(1)} \otimes a_{(2)} \;\;\; {\rm and} \;\;\; (\Delta \otimes {\rm id}_B) \circ \Delta(a) = a_{(1)} \otimes a_{(2)} \otimes a_{(3)} =  ({\rm id}_B \otimes \Delta) \circ \Delta(a).
\end{equation*}
In order to avoid ambiguity, we write \( \Delta_B = \Delta \) and \( \varepsilon_B = \varepsilon \) as needed. The biopposite weak bialgebra \( B^{bop} \) of \( B \) can be defined similar to the ordinary bialgebra.

Let \( B^{\prime} \) be a weak bialgebra. A \( \mathbb{K} \)-linear map \( f \colon B \to B^{\prime} \) is called a weak bialgebra homomorphism if \( f \) is a \( \mathbb{K} \)-algebra and \( \mathbb{K} \)-coalgebra homomorphism.
\end{defi}

We introduce two maps \( \varepsilon_s \), \( \varepsilon_t \colon B \to B \) defined by
\begin{align}
\varepsilon_s( a ) &= 1_{(1)} \varepsilon( a 1_{(2)} ); \\
\varepsilon_t( a ) &= \varepsilon( 1_{(1)} a ) 1_{(2)} \;\; (a \in B).
\end{align}
These maps \( \varepsilon_s \) and \( \varepsilon_t \) are respectively called the source counital map and the target counital map. 

\begin{lemm}
The maps \( \varepsilon_s \) and \( \varepsilon_t \) satisfy
\begin{align}
&\varepsilon_s \circ \varepsilon_s = \varepsilon_s, \;\;\; \varepsilon_t \circ \varepsilon_t = \varepsilon_t; \label{lem:estid} \\
&\varepsilon_s( 1_B ) = \varepsilon_t( 1_B ) = 1_B. \label{lem:estu}
\end{align}
\end{lemm}

\begin{lemm}
For an arbitrary element \( a \) in a weak bialgebra \( B \), 
\begin{align}
&1_{(1)} \varepsilon_s( a 1_{(2)} ) = \varepsilon_s( a ), \;\;\; \varepsilon_t( 1_{(1)} a ) 1_{(2)} = \varepsilon_t( a ); \\
&\Delta( \varepsilon_s( a ) ) = 1_{(1)} \otimes \varepsilon_s( a ) 1_{(2)} = 1_{(1)} \otimes 1_{(2)} \varepsilon_s( a ); \label{lem:coes} \\ 
&\Delta( \varepsilon_t( a ) ) = \varepsilon_t( a ) 1_{(1)} \otimes 1_{(2)} = 1_{(1)} \varepsilon_t( a ) \otimes 1_{(2)}; \label{lem:coet} \\
&a_{(1)} \otimes \varepsilon_s( a_{(2)} ) = a 1_{(1)} \otimes \varepsilon_s( 1_{(2)} ), \;\; \varepsilon_t( a_{(1)} ) \otimes a_{(2)} = \varepsilon_t( 1_{(1)} ) \otimes 1_{(2)} a; \label{lem:estd2} \\
&\varepsilon_s( a_{(1)} ) \otimes a_{(2)} = 1_{(1)} \otimes a 1_{(2)}, \;\; a_{(1)} \otimes \varepsilon_t( a_{(2)} ) = 1_{(1)} a \otimes 1_{(2)}. \label{lem:estd1}
\end{align}
\end{lemm}

\begin{lemm}
We denote by \( B \) a weak bialgebra. For any \( a, b \in B \),
\begin{align}
&\varepsilon_s( a ) \varepsilon_t( b ) = \varepsilon_t( b ) \varepsilon_s( a ); \label{lem:estcomm} \\
&\varepsilon( a b ) = \varepsilon( \varepsilon_s( a ) b ), \;\;\; \varepsilon( ab ) = \varepsilon( a \varepsilon_t( b ) ); \label{lem:eest2} \\
&\varepsilon_s( a b ) = \varepsilon_s( \varepsilon_s( a ) b ), \;\;\; \varepsilon_t( ab ) = \varepsilon_t( a \varepsilon_t( b ) ); \label{lem:est2} \\
&\varepsilon_s( a ) b = b_{(1)} \varepsilon_s( a b_{(2)} ), \;\;\; a \varepsilon_t( b ) = \varepsilon_t( a_{(1)} b ) a_{(2)}; \label{lem:esco} \\
&\varepsilon_s( a ) \varepsilon_s( b ) = \varepsilon_s( a \varepsilon_s( b ) ), \;\;\; \varepsilon_t( a ) \varepsilon_t( b ) = \varepsilon_t( \varepsilon_t( a ) b ). \label{lem:estm}
\end{align}
\end{lemm}

\begin{defi}
A weak bialgebra \( H \) with an \( \mathbb{K} \)-linear map \( S \colon H \to H \) is called a weak Hopf algebra, iff
\begin{align}
S( h_{(1)} ) h_{(2)} = \varepsilon_s( h ); \\
h_{(1)} S( h_{(2)} ) = \varepsilon_t( h ); \\
S( h_{(1)} ) h_{(2)} S( h_{(3)} ) = S( h )
\end{align}
are satisfied for all \( h \in H \). This \( S \), also called an antipode, is unique if there exists. 
\end{defi}
If there is a possibility of confusion, we write \( S^{{\rm WHA}} \) for the antipode of a weak Hopf algebra and \( S^{{\rm HAD}} \) for an antipode of a Hopf algebroid.
                                                                                                                                                                                                                                                                                                                                                                                                                                                                                                                                                                                                                                                                                                                                                                                                                                                                         
Let us recall the notion of Frobenius-separable \( \mathbb{K} \)-algebras to discuss relations between the left bialgebroid and the weak bialgebra.
A Frobenius-separable \( \mathbb{K} \)-algebra is a \( \mathbb{K} \)-algebra \( L \) equipped with a \( \mathbb{K} \)-linear map \( \psi \colon L \to \mathbb{K} \) and an element \( e^{(1)} \otimes e^{(2)} \in L \otimes_{\mathbb{K}} L \) such that
\begin{align}
l = \psi( l e^{(1)} ) e^{(2)} = e^{(1)} \psi( e^{(2)} l ), \;\; e^{(1)} e^{(2)} = 1_L \;\; ( \forall l \in L ).
\end{align}
This pair \( (\psi, e^{(1)} \otimes e^{(2)}) \) is called an idempotent Frobenius system.

\begin{prop}(See \cite[Theorem 5.5]{schau}.) \label{prop:LWF}
Let \( \mathcal{A}_L = (A, L, s_L, t_L, \Delta_L, \pi_L) \) be a left bialgebroid. If the base algebra \( L \) is a Frobenius-separable \( \mathbb{K} \)-algebra with an idempotent Frobenius system \( (\psi, e^{(1)} \otimes e^{(2)}) \), then the total algebra \( A \) has the following weak bialgebra structure \( (A, \Delta, \varepsilon) \):
\begin{align}
\Delta( a ) &= t_L(e^{(1)}) a_{[1]} \otimes s_L(e^{(2)}) a_{[2]}; \\
\varepsilon( a ) &= \psi \circ \pi_L( a ) \;\; (a \in A).
\end{align}
\end{prop}

Under the conditions of Proposition \ref{prop:LWF}, we suppose that the left bialgebroid \( \mathcal{A}_L \) has an antipode \( S^{{\rm HAD}} \). Then it is important to discuss whether the weak bialgebra \( A \) becomes a weak Hopf algebra or not. Schauenburg \cite{schau} solved this problem when \( \mathcal{A}_L \) is a \( \times_L \)-Hopf algebra, which is a generalization of the Hopf algebroid. We briefly sketch a special case of Corollary 6.2 in \cite{schau}.

For the total algebra \( A \) of a Hopf algebroid \( ( \mathcal{A}_L, S^{{\rm HAD}} ) \), we can define another left \( N \)-module structure \( _N A \) by
\begin{equation}
_N A \colon n \cdot a = S^{\mathrm{HAD}} \circ s_L \circ \omega^{-1}( n ) a \;\; (a \in A, n \in N).
\end{equation}
Then the tensor product \( A \otimes_N A \) has two meanings depending on left actions \( ^N A \) and \( _N A \). In order to avoid misunderstandings, we specify these actions by \( A ^N \otimes^N A \) and \( A ^N \otimes_N A \). For example, the tensor product \( A \otimes_N A \) in \eqref{SAL} and \eqref{SAN} stands for \( A ^N \otimes^N A \).

\begin{prop}(See \cite[Proposition 4.2(iv)]{bosz}.) \label{prop:bija}
Let \( (\mathcal{A}_L, S^{\mathrm{HAD}}) \) be a Hopf algebroid. Then the following \( \mathbb{K} \)-linear map \( \alpha \) is bijective with the inverse \( \alpha^{-1} \).
\begin{align}
&\alpha \colon A^N \otimes_N A \ni a \otimes b \mapsto a_{[1]} \otimes a_{[2]} b \in  A \otimes_L A; \\
&\alpha^{-1} \colon A \otimes_L A \ni a \otimes b \mapsto a^{[1]} \otimes S^{{\rm HAD}}( a^{[2]} ) b \in A^N \otimes_N A. \label{alpha-1}
\end{align}
These maps make sense by virtue of \( \eqref{St} \). Here we write \( \Delta_N( a ) = a^{[1]} \otimes a^{[2]} \).
\end{prop}

\begin{prop}(See \cite[Corollary 6.2]{schau}.) \label{prop:WHD}
Let \( \mathcal{A}_L = (A, L, s_L, t_L, \Delta_L, \pi_L) \) be a left bialgebroid satisfying the conditions of Proposition \ref{prop:LWF}. If \( (\mathcal{A}_L, S^{{\rm HAD}}) \) is a Hopf algebroid, then the total algebra \( A \) becomes a weak Hopf algebra whose antipode \( S^{\mathrm{WHA}} \) is defined by 
\begin{equation}
S^{{\rm WHA}}( a ) = \varepsilon_s( a^{[1]} ) S^{{\rm HAD}}( a^{[2]} ) \;\; (a \in A).
\end{equation}
This \( S^{{\rm WHA}} \) makes sense because of \( \alpha^{-1} \) and the following \( \mathbb{K} \)-linear map:
\begin{align}
&\beta \colon A^N \otimes_N A \ni a \otimes b \mapsto \varepsilon_s( a ) b \in A. \label{beta}
\end{align}
This \( \beta \) is well defined because of \eqref{def:wdm} and \eqref{lem:esco}.
\end{prop} 

\section{Two left bialgebroids \( \mathfrak{A}(w) \) and \( A_{\sigma} \)} \label{sec:TL}

\subsection{Summary of left bialgebroid \( A_{\sigma} \)} \label{sec:As}

In this subsection, we recall a left bialgebroid \( A_{\sigma} \). For more details, we refer to \cite{oshibu}. This is a generalization of \cite{shiburi}.

Let \( R \) be a \( \mathbb{K} \)-algebra and \( \Lambda \) a non-empty finite set. Let \(G\) denote the opposite group of the symmetric group on the set \( \Lambda \). We can define a right group action of this group \( G \) on the set \( \Lambda \):\( \lambda \alpha = \alpha(\lambda) \; (\lambda \in \Lambda, \alpha \in G) \). We denote by \( M_{\Lambda}(R) \) the \( \mathbb{K} \)-algebra consisting of maps from \( \Lambda \) to \( R \). For any \( \alpha \in G \), the map \( T_{\alpha} : M_{\Lambda}(R) \to M_{\Lambda}(R) \) is defined by
\begin{equation*}
T_{\alpha}(f)(\lambda) = f(\lambda \alpha) \;\; (f \in M_{\Lambda}(R), \lambda \in \Lambda).
\end{equation*}
The map \( T_{\alpha} \; (\alpha \in G) \) is a \( \mathbb{K} \)-algebra homomorphism such that \( T_{\alpha} \circ T_{\alpha^{-1}} = {\rm id}_{M_{\Lambda}(R)} \). Let \( \deg \) be a map from a finite set \( X \) to the group \( G \). Define
\begin{align*}
\Lambda X := ( M_{\Lambda}(R) \otimes_{\mathbb{K}} M_{\Lambda}(R)^{op} ) \bigsqcup \{ L_{ab} \; | \; a, b \in X \} \bigsqcup \{ (L^{-1})_{ab} \; | \; a, b \in X \}.
\end{align*}
Let \( \sigma^{ab}_{cd} \in M_{\Lambda}(R) \; (a, b, c, d \in X) \) and we denote by \( \mathbb{K} \langle \Lambda X \rangle \) the free \( \mathbb{K} \)-algebra generated by the set \( \Lambda X \). The symbol \( I_{\sigma} \) means the two-sided ideal of \( \mathbb{K} \langle \Lambda X \rangle \) whose generators are: 

\begin{itemize}

\item[(1)]
\(\xi+\xi^{\prime} - (\xi+\xi^{\prime}),\;c\xi - (c\xi),\;\xi\xi^{\prime} - (\xi\xi^{\prime})\;\;(\forall c\in\mathbb{K},\forall \xi,\xi^{\prime}\in M_{\Lambda}(R)\otimes_{\mathbb{K}} M_{\Lambda}(R)^{op}).\) \\
Here the notation \( \xi+\xi^{\prime} \) stands for the addition in the algebra \( \mathbb{K} \langle \Lambda X \rangle \), while the notation \( ( \xi+\xi^{\prime} ) ( \in \Lambda X ) \) is that of the algebra \( M_{\Lambda}(R) \otimes_{\mathbb{K}} M_{\Lambda}(R)^{op} \). The other two generators for the scalar multiplication and multiplication are similar.

\item[(2)]
\(\displaystyle\sum_{c\in X} L_{ac}(L^{-1})_{cb}-\delta_{a, b} \emptyset ,\;\sum_{c\in X} (L^{-1})_{ac}L_{cb}-\delta_{a, b} \emptyset \;\; (\forall a,b\in X). \) \\
Here \( \delta_{a, b} \in \mathbb{K} \; (a, b \in X) \) means Kronecker's delta symbol and \( \emptyset \) means the empty word.

\item[(3)]
\((T_{\deg (a)}(f)\otimes1_{M_{\Lambda}(R)})L_{ab} - L_{ab}(f\otimes1_{M_{\Lambda}(R)}), \\
   (1_{M_{\Lambda}(R)} \otimes T_{\deg (b)}(f))L_{ab} - L_{ab}(1_{M_{\Lambda}(R)}\otimes f), \\
   (f\otimes1_{M_{\Lambda}(R)})(L^{-1})_{ab} - (L^{-1})_{ab}(T_{\deg (b)}(f)\otimes1_{M_{\Lambda}(R)}), \\
   (1_{M_{\Lambda}(R)}\otimes f)(L^{-1})_{ab} - (L^{-1})_{ab}(1_{M_{\Lambda}(R)}\otimes T_{\deg (a)}(f))\;\;(\forall f\in M_{\Lambda}(R), \forall a,b\in X).\) \\

\item[(4)]
\(\displaystyle\sum_{x,y\in X} (\sigma^{xy}_{ac}\otimes1_{M_{\Lambda}(R)})L_{yd}L_{xb} - \sum_{x,y\in X} (1_{M_{\Lambda}(R)}\otimes\sigma^{bd}_{xy})L_{cy}L_{ax}\;\;(\forall a,b,c,d\in X).\) \\

\item[(5)]
\(\emptyset - 1_{M_{\Lambda}(R)}\otimes1_{M_{\Lambda}(R)}\). \\
\end{itemize}

\begin{theo}(See \cite[Theorem 2.1]{oshibu}.)
If the following conditions are satisfied, then the quotient \( A_{\sigma} := \mathbb{K} \langle \Lambda X \rangle / I_{\sigma} \) is a left bialgebroid.
\begin{equation}
\begin{aligned}
\begin{cases}
\sigma^{ab}_{cd} (\lambda) \in Z(R) \;\; ( \forall \lambda \in \Lambda, \forall a, b, c, d \in X ) ; \\
\lambda \deg(d) \deg(b) \neq \lambda \deg(c) \deg(a) \Rightarrow \sigma^{bd}_{ac} (\lambda) = 0. \label{cond:invc}
\end{cases}
\end{aligned}
\end{equation}
Here \( Z(R) \) is the center of \( R \).
\end{theo}

The maps \( s_{M_{\Lambda}(R)} \colon M_{\Lambda}(R) \to A_{\sigma} \) and \( t_{M_{\Lambda}(R)} \colon M_{\Lambda}(R)^{op} \to A_{\sigma} \) are defined by
\begin{align*}
&s_{M_{\Lambda}(R)}(f) = f \otimes 1_{M_{\Lambda}(R)} + I_{\sigma}; \\
&t_{M_{\Lambda}(R)}(f) = 1_{M_{\Lambda}(R)} \otimes f + I_{\sigma} \;\; (f \in M_{\Lambda}(R)).
\end{align*}
These are \( \mathbb{K} \)-algebra homomorphisms and satisfy \eqref{def:stcom}. Thus \( A_{\sigma} \) is an \( (M_{\Lambda}(R), M_{\Lambda}(R)) \)-bimodule via \eqref{def:ac}.

Let \( I_2 \) denote the right ideal of \( A_{\sigma} \) whose generators are \( t_{M_{\Lambda}(R)}( f ) \otimes 1_{A_{\sigma}} - 1_{A_{\sigma}} \otimes s_{M_{\Lambda}(R)}( f ) \; (\forall f \in M_{\Lambda}(R)) \). The \( \mathbb{K} \)-algebra homomorphism \( \overline{\Delta} \colon \mathbb{K} \langle \Lambda X \rangle \to A_{\sigma} \otimes_{\mathbb{K}} A_{\sigma} \) is defined by
\begin{align*}
&\overline{\Delta}(\xi) = s_{M_{\Lambda}(R)} \otimes t_{M_{\Lambda}(R)}(\xi) \;\; (\xi \in M_{\Lambda}(R) \otimes_{\mathbb{K}} M_{\Lambda}(R)^{op}); \\
&\overline{\Delta}(L_{ab}) = \sum_{c \in X} L_{ac} + I_{\sigma} \otimes L_{cb} + I_{\sigma} \;\; (a,b \in X); \\
&\overline{\Delta}((L^{-1})_{ab}) = \sum_{c \in X} (L^{-1})_{cb} + I_{\sigma} \otimes (L^{-1})_{ac} + I_{\sigma}.
\end{align*}
This map \( \overline{\Delta} \) satisfies \( \overline{\Delta}( I_{\sigma} ) \subset I_2 \). Thus the \( \mathbb{K} \)-linear map \( \tilde{\Delta}( \alpha + I_{\sigma} ) = \overline{\Delta}( \alpha ) + I_2 \; (\alpha \in \mathbb{K} \langle \Lambda X \rangle) \) is well defined. Since \( A_{\sigma} \otimes_{\mathbb{K}} A_{\sigma} / I_2 \cong A_{\sigma} \otimes_{M_{\Lambda}(R)} A_{\sigma} \) as \( \mathbb{K} \)-vector spaces, we can induce the \( \mathbb{K} \)-linear map \( \Delta_{M_{\Lambda}(R)} \colon A_{\sigma} \to A_{\sigma} \otimes_{M_{\Lambda}(R)} A_{\sigma} \) from the map \( \tilde{\Delta} \). This \( \Delta_{M_{\Lambda}(R)} \) is an \( (M_{\Lambda}(R), M_{\Lambda}(R)) \)-bimodule homomorphism. 

The next is to define the map \( \pi_{M_{\Lambda}(R)} \colon A_{\sigma} \to M_{\Lambda}(R) \). The \( \mathbb{K} \)-algebra homomorphism \( \overline{\chi} \colon \mathbb{K} \langle \Lambda X \rangle \to {\rm End}_{\mathbb{K}}( M_{\Lambda}(R) ) \) is defined by
\begin{align*}
&\overline{\chi}(f \otimes g) = \rho_l( f ) \rho_r(g) \; (f, g \in M_{\Lambda}(R)) \\
&\overline{\chi} ( L_{ab} ) = \delta_{a, b} T_{\deg (a)}; \\
&\overline{\chi} ( (L^{-1})_{ab} ) = \delta_{a, b} T_{\deg (a)^{-1}} \;\; ( a, b \in X )
\end{align*}
Here \( \rho_l (f) \) and \( \rho_r (f) \; (f \in M_{\Lambda}(R)) \) are maps defined by
\begin{equation*}
\rho_l(f) \colon M_{\Lambda}(R) \ni g \mapsto fg \in M_{\Lambda}(R); \;\; \rho_r(f) \colon M_{\Lambda}(R) \ni g \mapsto gf \in M_{\Lambda}(R).
\end{equation*}

Because \( \overline{\chi}( I_{\sigma} ) = \{ 0 \} \) is satisfied, the map \( \chi(\alpha + I_{\sigma}) = \overline{\chi}(\alpha) \; (\alpha \in \mathbb{K} \langle \Lambda X \rangle) \) makes sense and is a \( \mathbb{K} \)-algebra homomorphism. We define the map \( \pi_{M_{\Lambda}(R)} \) by
\begin{equation}
\pi_{M_{\Lambda}(R)} \colon A_{\sigma} \ni a \mapsto \chi(a)( 1_{M_{\Lambda}(R)} ) \in M_{\Lambda}(R).
\end{equation}
This \( \pi_{M_{\Lambda}(R)} \) is an \( (M_{\Lambda}(R), M_{\Lambda}(R)) \)-bimodule homomorphism.

The triplet \( (A_{\sigma}, \Delta_{M_{\Lambda}(R)}, \pi_{M_{\Lambda}(R)}) \) is a comonoid in the tensor category of \( (M_{\Lambda}(R), M_{\Lambda}(R)) \)-bimodules. Since the maps \( \Delta_{M_{\Lambda}(R)} \) and \( \pi_{M_{\Lambda}(R)} \) satisfy the conditions \eqref{def:st} - \eqref{def:pmulti},  the sextuplet \( (A_{\sigma}, M_{\Lambda}(R), s_{M_{\Lambda}(R)}, t_{M_{\Lambda}(R)}, \Delta_{M_{\Lambda}(R)}, \pi_{M_{\Lambda}(R)}) \) is a left bialgebroid.

Let \( \sigma = \{ \sigma^{ab}_{cd} \}_{a, b, c, d \in X} \). This left bialgebroid \( A_{\sigma} \) has a Hopf algebroid structure if \( \sigma \) satisfies a certain condition, called rigidity.

\begin{defi}(See \cite[Definition 4.2]{oshibu}.)
The family \( \sigma = \{ \sigma^{ab}_{cd} \}_{a, b, c, d \in X} \) is called rigid, iff, for any \( a, b \in X \), there exist \( x_{ab}, y_{ab} \in A_{\sigma} \) such that
\begin{align*}
\sum_{c \in X} ( (L^{-1})_{cb} + I_{\sigma} ) x_{ac} &= \sum_{c \in X} x_{cb} ( (L^{-1})_{ac} + I_{\sigma} ) \\
&= \sum_{c \in X} ( L_{cb} + I_{\sigma} ) y_{ac} \\
&= \sum_{c \in X} y_{cb} ( L_{ac} + I_{\sigma} ) \\
&= \delta_{a, b} 1_{A_{\sigma}}.
\end{align*}
\end{defi}

\begin{prop}(See \cite[Proposition 4.1]{oshibu}.)
The following are equivalent:
\begin{enumerate}
\item \( \sigma \) is rigid;
\item There exists a unique \( \mathbb{K} \)-algebra anti-automorphism \( S \colon A_{\sigma} \to A_{\sigma} \) satisfying
\begin{equation}
\begin{cases}
	S(f \otimes g + I_{\sigma}) = g \otimes f + I_{\sigma} \; (f, g \in M_{\Lambda}(R)); \\
	S(L_{ab} + I_{\sigma}) = (L^{-1})_{ab} + I_{\sigma}; \\
	S((L^{-1})_{ab} + I_{\sigma}) = x_{ab} \;\; (a, b \in X).
\end{cases}
\end{equation} 
\end{enumerate}
\end{prop}

\begin{prop}(See \cite[Proposition 4.2]{oshibu}.)
If \( \sigma \) is rigid, then the pair \( (A_{\sigma}, S) \) is a Hopf algebroid for \( N = M_{\Lambda}( R )^{op} \) and \( \omega = {\rm id}_{M_{\Lambda}( R )} \).
\end{prop}

\subsection{Left bialgebroid \( \mathfrak{A}(w) \)} \label{sec:Aw}

In this subsection, we introduce a left bialgebroid \( \mathfrak{A}(w) \). This is a generalization of \cite{hayas}.

\begin{defi}
Let \( \Lambda \) be a non-empty set. A set \(Q\) endowed with two maps \( \mathfrak{s}, \mathfrak{t} \colon Q \to \Lambda \) is said to be a quiver over \( \Lambda \). These maps \( \mathfrak{s} \) and \( \mathfrak{t} \) are respectively called the source map and the target map.

For a non-negative integer \( m \), we define the fiber product \( Q^{(m)} \) by \( Q^{(0)} := \Lambda \), \( Q^{(1)} := Q \), and \( Q^{(m)} := \{ q = (q_1, \ldots, q_m) \in Q^m \; | \; \mathfrak{t}( q_{i} ) = \mathfrak{s}( q_{i + 1} ), 1 \leq \forall i \leq m-1 \} \; (m>1) \). The set \( Q^{(m)} \; (m>0) \) is a quiver over \( \Lambda \) with \( \mathfrak{s}(q) = \mathfrak{s}(q_1) \), \( \mathfrak{t}(q) = \mathfrak{t}(q_m) \). \( Q^{(0)} \) is also a quiver over \( \Lambda \) by \( \mathfrak{s} = \mathfrak{t} = {\rm id}_{\Lambda} \).
\end{defi}

Let \( \Lambda \) be a non-empty finite set, and \( Q \) a finite quiver over \( \Lambda \). We denote by \( \mathfrak{G}(Q) \) the linear span of the symbols \( \displaystyle \mathbf{e}\genfrac{[}{]}{0pt}{}{p}{q} \) \( ( p, q \in Q^{(m)}, m \in \mathbb{Z}_{\geq 0} ) \): 

\begin{equation}
\mathfrak{G}(Q) := \bigoplus_{p, q \in Q^{(m)}, m \in \mathbb{Z}_{\geq 0}} \mathbb{K} \; \mathbf{e}\genfrac{[}{]}{0pt}{}{p}{q}.
\end{equation}
This \( \mathfrak{G}(Q) \) is a \( \mathbb{K} \)-algebra by the following multiplication:

\begin{align*}
\mathbf{e}\genfrac{[}{]}{0pt}{}{p}{q} \mathbf{e}\genfrac{[}{]}{0pt}{}{p^{\prime}}{q^{\prime}} &= \delta_{\mathfrak{t}( p ), \mathfrak{s}( p^{\prime} )} \delta_{\mathfrak{t}( q ), \mathfrak{s}( q^{\prime} )}
\mathbf{e}\genfrac{[}{]}{0pt}{}{p p^{\prime}}{q q^{\prime}}; \\
1_{\mathfrak{G}(Q)} &= \sum_{\lambda, \mu \in \Lambda} \mathbf{e}\genfrac{[}{]}{0pt}{}{\lambda}{\mu}
\end{align*}
for \( p, q \in Q^{(m)}, p^{\prime}, q^{\prime} \in Q^{(n)} \), and \( m, n \in \mathbb{Z}_{\geq 0} \). Here \( \delta_{\lambda, \mu} \in \mathbb{K} \; (\lambda, \mu \in \Lambda) \) means Kronecker's delta symbol. For a \( \mathbb{K} \)-algebra \(R\), let \( \mathbf{w}\begin{sumibmatrix} & a & \\ c & & b \\ & d & \end{sumibmatrix} \in R \; ((a, b), (c, d) \in Q^{(2)}) \).
We write \( \mathfrak{I}_{\mathbf{w}} \) for the two-sided ideal of the \( \mathbb{K} \)-algebra \( \mathfrak{H}(Q) := R \otimes_{\mathbb{K}} R^{op} \otimes_{\mathbb{K}} \mathfrak{G}(Q) \) whose generators are
\begin{align}
&\sum_{( x, y ) \in Q^{(2)}} \mathbf{w}\begin{sumibmatrix} & x & \\ a & & y \\ & b & \end{sumibmatrix} \otimes 1_R \otimes \mathbf{e}\genfrac{[}{]}{0pt}{}{x}{c} \mathbf{e}\genfrac{[}{]}{0pt}{}{y}{d} \nonumber \\
&- \sum_{( x, y ) \in Q^{(2)}} 1_R \otimes \mathbf{w}\begin{sumibmatrix} & c & \\ x & & d \\ & y & \end{sumibmatrix} \otimes \mathbf{e}\genfrac{[}{]}{0pt}{}{a}{x} \mathbf{e}\genfrac{[}{]}{0pt}{}{b}{y} \;\; ( \forall (a, b), (c, d) \in Q^{(2)} ). \label{gen:face}
\end{align}
We define \( \mathfrak{A}(w) \) by the quotient \( \mathfrak{A}(w) := \mathfrak{H}(Q) / \mathfrak{I}_{\mathbf{w}} \).

\begin{theo}\label{theo:lb}
If the following conditions are satisfied, then \( \mathfrak{A}(w) \) is a left bialgebroid.
\begin{equation}
\begin{aligned}
\begin{cases}
\mathbf{w}\begin{sumibmatrix} & a & \\ c & & b \\ & d & \end{sumibmatrix} \in Z(R) \;\; ( \forall (a, b) , (c, d) \in Q^{(2)} ) ; \\
\mathfrak{s}( a ) \neq \mathfrak{s}( c ) \; {\rm or} \; \mathfrak{t}( b ) \neq \mathfrak{t}( d ) \Rightarrow \mathbf{w}\begin{sumibmatrix} & a & \\ c & & b \\ & d & \end{sumibmatrix} = 0.
\end{cases}
\end{aligned}
\label{face}
\end{equation}
\end{theo}

The maps \( s_{M_{\Lambda}(R)} \colon M_{\Lambda}(R) \to \mathfrak{A}(w) \) and \( t_{M_{\Lambda}(R)} \colon M_{\Lambda}(R)^{op} \to \mathfrak{A}(w) \) are defined by
\begin{align*}
&s_{M_{\Lambda}(R)}(f) = \sum_{\lambda, \mu \in \Lambda} f( \lambda ) \otimes 1_R \otimes \mathbf{e}\genfrac{[}{]}{0pt}{}{\lambda}{\mu} + \mathfrak{I}_{\mathbf{w}} ; \\
&t_{M_{\Lambda}(R)}(f) = \sum_{\lambda, \mu \in \Lambda} 1_R \otimes f( \lambda ) \otimes \mathbf{e}\genfrac{[}{]}{0pt}{}{\mu}{\lambda} + \mathfrak{I}_{\mathbf{w}} \;\; (f \in M_{\Lambda}(R)).
\end{align*}
These maps are \( \mathbb{K} \)-algebra homomorphisms satisfying \eqref{def:stcom}. As a result, \( \mathfrak{A}(w) \) is an \( (M_{\Lambda}(R), M_{\Lambda}(R)) \)-bimodule by the action \eqref{def:ac}.

Let \( \mathfrak{I}_2 \) denote the right ideal of \( \mathfrak{A}(w) \otimes_{\mathbb{K}} \mathfrak{A}(w) \) whose generators are \( t_{M_{\Lambda}(R)}( f ) \otimes 1_{\mathfrak{A}(w)} - 1_{\mathfrak{A}(w)} \otimes s_{M_{\Lambda}(R)}( f ) \; (\forall f \in M_{\Lambda}(R)) \).
In order to construct the map \( \Delta_{M_{\Lambda}(R)} \), we define the \( \mathbb{K} \)-linear map \( \overline{\nabla} \colon \mathfrak{H}(Q) \to \mathfrak{A}(w) \otimes_{\mathbb{K}} \mathfrak{A}(w) \) by
\begin{align*}
\overline{\nabla}( r \otimes r^{\prime} \otimes \mathbf{e}\genfrac{[}{]}{0pt}{}{p}{q} ) = \sum_{u \in Q^{(m)}} ( r \otimes 1_R \otimes \mathbf{e}\genfrac{[}{]}{0pt}{}{p}{u} + \mathfrak{I}_{\mathbf{w}} ) \otimes ( 1_R \otimes r^{\prime} \otimes \mathbf{e}\genfrac{[}{]}{0pt}{}{u}{q} + \mathfrak{I}_{\mathbf{w}} )
\end{align*}
for \( r, r^{\prime} \in R \), \( p, q \in Q^{(m)} \), and \( m \in \mathbb{Z}_{\geq 0} \). This map \( \overline{\nabla} \) preserves the multiplication of the \( \mathbb{K} \)-algebra \( \mathfrak{H}(Q) \).

\begin{prop}
\( \overline{\nabla}( \mathfrak{I}_{\mathbf{w}} ) \subset \mathfrak{I}_2 \).
\end{prop}

\begin{proof}
It is easy to check that \( \overline{\nabla}( \alpha ) \beta \in \mathfrak{I}_2 \) for any \( \alpha \in \mathfrak{H}(Q) \) and \( \beta \in \mathfrak{I}_2 \). In order to complete the proof, we need to show that \( \overline{\nabla}( \gamma ) \in \mathfrak{I}_2 \) for an arbitrary generator \( \gamma \) in \eqref{gen:face}.

For any \( (a,b), (c,d) \in Q^{(2)} \), we can induce the following equality by using the definition of \( \mathfrak{I}_{\mathbf{w}} \):
\begin{align*}
&\overline{\nabla}( \sum_{( x, y ) \in Q^{(2)}} \mathbf{w}\begin{sumibmatrix} & x & \\ a & & y \\ & b & \end{sumibmatrix} \otimes 1_R \otimes \mathbf{e}\genfrac{[}{]}{0pt}{}{x}{c} \mathbf{e}\genfrac{[}{]}{0pt}{}{y}{d}
- \sum_{( x, y ) \in Q^{(2)}} 1_R \otimes \mathbf{w}\begin{sumibmatrix} & c & \\ x & & d \\ & y & \end{sumibmatrix} \otimes \mathbf{e}\genfrac{[}{]}{0pt}{}{a}{x} \mathbf{e}\genfrac{[}{]}{0pt}{}{b}{y} ) \\
=& \sum_{(x, y), (u, v) \in Q^{(m)}} ( \mathbf{w}\begin{sumibmatrix} & x & \\ a & & y \\ & b & \end{sumibmatrix} \otimes 1_R \otimes \mathbf{e}\genfrac{[}{]}{0pt}{}{x}{u} \mathbf{e}\genfrac{[}{]}{0pt}{}{y}{v} + \mathfrak{I}_{\mathbf{w}} ) \otimes ( 1_R \otimes 1_R \otimes \mathbf{e}\genfrac{[}{]}{0pt}{}{u}{c} \mathbf{e}\genfrac{[}{]}{0pt}{}{v}{d} + \mathfrak{I}_{\mathbf{w}} ) \\
-& \sum_{(x, y), (u, v) \in Q^{(m)}} ( 1_R \otimes 1_R \otimes \mathbf{e}\genfrac{[}{]}{0pt}{}{a}{u} \mathbf{e}\genfrac{[}{]}{0pt}{}{b}{v} + \mathfrak{I}_{\mathbf{w}} ) \otimes ( 1_R \otimes \mathbf{w}\begin{sumibmatrix} & c & \\ x & & d \\ & y & \end{sumibmatrix} \otimes \mathbf{e}\genfrac{[}{]}{0pt}{}{u}{x} \mathbf{e}\genfrac{[}{]}{0pt}{}{v}{y} + \mathfrak{I}_{\mathbf{w}} ) \\
=& \sum_{(x, y), (u, v) \in Q^{(m)}} ( 1_R \otimes \mathbf{w}\begin{sumibmatrix} & u & \\ x & & v \\ & y & \end{sumibmatrix} \otimes \mathbf{e}\genfrac{[}{]}{0pt}{}{a}{x} \mathbf{e}\genfrac{[}{]}{0pt}{}{b}{y} + \mathfrak{I}_{\mathbf{w}} ) \otimes ( 1_R \otimes 1_R \otimes \mathbf{e}\genfrac{[}{]}{0pt}{}{u}{c} \mathbf{e}\genfrac{[}{]}{0pt}{}{v}{d} + \mathfrak{I}_{\mathbf{w}} ) \\
-& \sum_{(x, y), (u, v) \in Q^{(m)}} ( 1_R \otimes 1_R \otimes \mathbf{e}\genfrac{[}{]}{0pt}{}{a}{x} \mathbf{e}\genfrac{[}{]}{0pt}{}{b}{y} + \mathfrak{I}_{\mathbf{w}} ) \otimes ( \mathbf{w}\begin{sumibmatrix} & u & \\ x & & v \\ & y & \end{sumibmatrix} \otimes 1_R \otimes \mathbf{e}\genfrac{[}{]}{0pt}{}{u}{c} \mathbf{e}\genfrac{[}{]}{0pt}{}{v}{d} + \mathfrak{I}_{\mathbf{w}} ) \\
=& \sum_{(x, y), (u, v) \in Q^{(m)}} ( t_{M_{\Lambda}(R)}( \mathbf{w} \begin{sumibmatrix} & u & \\ x & & v \\ & y & \end{sumibmatrix}_M ) \otimes 1_{\mathfrak{A}(w)} - 1_{\mathfrak{A}(w)} \otimes s_{M_{\Lambda}(R)}( \mathbf{w}\begin{sumibmatrix} & u & \\ x & & v \\ & y & \end{sumibmatrix}_M ) ) \\
&\times ( ( 1_R \otimes 1_R \otimes \mathbf{e}\genfrac{[}{]}{0pt}{}{a}{x} \mathbf{e}\genfrac{[}{]}{0pt}{}{b}{y} + \mathfrak{I}_{\mathbf{w}} ) \otimes ( 1_R \otimes 1_R \otimes \mathbf{e}\genfrac{[}{]}{0pt}{}{u}{c} \mathbf{e}\genfrac{[}{]}{0pt}{}{v}{d} + \mathfrak{I}_{\mathbf{w}} ) ) \\
&\in \mathfrak{I}_2.
\end{align*}
Here \( r_M \in M_{\Lambda}(R) \; (r \in R) \) is the map defined by \( r_M( \lambda ) = r \; (\lambda \in \Lambda) \). Thus this proposition is proved.
\end{proof}

This proposition induces a \( \mathbb{K} \)-linear map \( \tilde{\nabla}( \alpha + \mathfrak{I}_{\mathbf{w}} ) = \overline{\nabla}( \alpha ) + \mathfrak{I}_2 \; (\alpha \in \mathfrak{H}(Q)) \). Since \( \mathfrak{A}(w) \otimes_{\mathbb{K}} \mathfrak{A}(w) / \mathfrak{I}_2 \cong  \mathfrak{A}(w) \otimes_{M_{\Lambda}(R)} \mathfrak{A}(w) \) as \( \mathbb{K} \)-vector spaces, we can construct the \( \mathbb{K} \)-linear map \( \Delta_{M_{\Lambda}(R)} \colon \mathfrak{A}(w) \to \mathfrak{A}(w) \otimes_{M_{\Lambda}(R)} \mathfrak{A}(w) \) from the map \( \tilde{\nabla} \). This \( \Delta_{M_{\Lambda}(R)} \) is an \( (M_{\Lambda}(R), M_{\Lambda}(R)) \)-bimodule homomorphism.

The next task is to construct the map \( \pi_{M_{\Lambda}(R)} \colon \mathfrak{A}(w) \to M_{\Lambda}(R) \). We first define the \( \mathbb{K} \)-linear map \( \overline{\zeta} \colon \mathfrak{H}(Q) \to {\rm End}_{\mathbb{K}}( M_{\Lambda}(R) ) \) as follows:
\begin{align*}
\overline{\zeta}( r \otimes r^{\prime} \otimes {\bf e}\genfrac{[}{]}{0pt}{}{p}{q} ) ( f ) = \delta_{p, q} (r f( \mathfrak{t}( q ) ) r^{\prime})_M \delta_{\mathfrak{s}( q )} \;\; (f \in M_{\Lambda}(R)).
\end{align*}
Here \( \delta_{\lambda} \in M_{\Lambda}(R) \; (\lambda \in \Lambda) \) is the map defined by \(\delta_{\lambda}( \mu ) = \delta_{\lambda, \mu} \; (\mu \in \Lambda) \). This map \( \overline{\zeta} \) is a \( \mathbb{K} \)-algebra homomorphism.

\begin{prop}
\( \overline{\zeta}( \mathfrak{I}_{\mathbf{w}} ) = \{ 0 \} \).
\end{prop}
\begin{proof}
We denote by \( f \) an arbitrary element in \( M_{\Lambda}(R) \). By using the first condition in \eqref{face},
\begin{align*}
&\overline{\zeta}( \sum_{(x, y ) \in Q^{(2)}} {\bf w}\begin{sumibmatrix} & x & \\ a & & y \\ & b & \end{sumibmatrix} \otimes 1_R \otimes {\bf e}\genfrac{[}{]}{0pt}{}{x}{c} {\bf e}\genfrac{[}{]}{0pt}{}{y}{d} \\
&- \sum_{( x, y ) \in Q^{(2)}} 1_R \otimes {\bf w}\begin{sumibmatrix} & c & \\ x & & d \\ & y & \end{sumibmatrix} \otimes {\bf e}\genfrac{[}{]}{0pt}{}{a}{x} {\bf e}\genfrac{[}{]}{0pt}{}{b}{y} ) ( f ) \\
=& ({\bf w}\begin{sumibmatrix} & c & \\ a & & d \\ & b & \end{sumibmatrix} f( \mathfrak{t}( d ) ))_M \delta_{\mathfrak{s}(c)} - ( f( \mathfrak{t}( b ) {\bf w}\begin{sumibmatrix} & c & \\ a & & d \\ & b & \end{sumibmatrix} )_M \delta_{\mathfrak{s}(a)} \\
=& ({\bf w}\begin{sumibmatrix} & c & \\ a & & d \\ & b & \end{sumibmatrix} f( \mathfrak{t}( d ) ))_M \delta_{\mathfrak{s}(c)} - ({\bf w}\begin{sumibmatrix} & c & \\ a & & d \\ & b & \end{sumibmatrix} f( \mathfrak{t}( b ) ))_M \delta_{\mathfrak{s}(a)}
\end{align*}
for all \( (a, b) \) and \( (c, d) \in Q^{(2)} \). If \( {\bf w}\begin{sumibmatrix} & c & \\ a & & d \\ & b & \end{sumibmatrix} \neq 0 \), \( \mathfrak{s}(a) = \mathfrak{s}(c) \) is satisfied because of the second condition in \eqref{face}. This completes the proof.
\end{proof}

As a result of this proposition, the map \( \zeta(\alpha + \mathfrak{I}_{\mathbf{w}}) = \overline{\zeta}(\alpha) \; (\alpha \in \mathfrak{H}(Q)) \) is an well defined \( \mathbb{K} \)-algebra homomorphism. We define the map \( \pi_{M_{\Lambda}(R)} \) by
\begin{equation}
\pi_{M_{\Lambda}(R)} \colon \mathfrak{A}(w) \ni a \mapsto \zeta(a)( 1_{M_{\Lambda}(R)} ) \in M_{\Lambda}(R).
\end{equation}
This \( \pi_{M_{\Lambda}(R)} \) is an \( (M_{\Lambda}(R), M_{\Lambda}(R)) \)-bimodule homomorphism.

\begin{prop}
The triplet \( (\mathfrak{A}(w), \Delta_{M_{\Lambda}(R)}, \pi_{M_{\Lambda}(R)}) \) is a comonoid in the tensor category of \( (M_{\Lambda}(R), M_{\Lambda}(R)) \)-bimodules.
\end{prop}

\begin{proof}
For any \( r, r^{\prime} \in R \), \( p, q \in Q^{(m)} \) and \( m \in \mathbb{Z}_{\geq 0} \),
\begin{align*}
&(\Delta_{M_{\Lambda}(R)} \otimes {\rm id_{\mathfrak{A}(w)}}) \circ \Delta_{M_{\Lambda}(R)}( r \otimes r^{\prime} \otimes \mathbf{e}\genfrac{[}{]}{0pt}{}{p}{q} + \mathfrak{I}_{\mathbf{w}} ) \\
=& \sum_{u, v \in Q^{(m)}} (r \otimes 1_R \otimes \mathbf{e}\genfrac{[}{]}{0pt}{}{p}{v} + \mathfrak{I}_{\mathbf{w}}) \otimes  (1_R \otimes 1_R \otimes \mathbf{e}\genfrac{[}{]}{0pt}{}{v}{u} + \mathfrak{I}_{\mathbf{w}}) \otimes (1_R \otimes r^{\prime} \otimes \mathbf{e}\genfrac{[}{]}{0pt}{}{u}{q} + \mathfrak{I}_{\mathbf{w}}) \\
=& ( {\rm id_{\mathfrak{A}(w)}} \otimes \Delta_{M_{\Lambda}(R)} ) \circ \Delta_{M_{\Lambda}(R)}( r \otimes r^{\prime} \otimes \mathbf{e}\genfrac{[}{]}{0pt}{}{p}{q} + \mathfrak{I}_{\mathbf{w}} ).
\end{align*}
We write \( \displaystyle a = r \otimes r^{\prime} \otimes \mathbf{e}\genfrac{[}{]}{0pt}{}{p}{q} + \mathfrak{I}_{\mathbf{w}} \). By using Sweedler's notation \( \Delta_{M_{\Lambda}(R)}(a) = a_{(1)} \otimes a_{(2)} \), 
\begin{align*}
\pi_{M_{\Lambda}(R)}(a_{(1)}) a_{(2)} =& \sum_{\substack{u \in Q^{(m)} \\ \lambda, \mu \in \Lambda}} \delta_{p, u} (r_M \delta_{\mathfrak{s}(u)}(\lambda) \otimes 1_R \otimes \mathbf{e}\genfrac{[}{]}{0pt}{}{\lambda}{\mu}) (1_R \otimes r^{\prime} \otimes \mathbf{e}\genfrac{[}{]}{0pt}{}{u}{q}) + \mathfrak{I}_{\mathbf{w}} \\
=& \sum_{u \in Q^{(m)}} \delta_{p,u} (r \otimes r^{\prime} \otimes \mathbf{e}\genfrac{[}{]}{0pt}{}{u}{q}) + \mathfrak{I}_{\mathbf{w}} \\
=& r \otimes r^{\prime} \otimes \mathbf{e}\genfrac{[}{]}{0pt}{}{p}{q} + \mathfrak{I}_{\mathbf{w}}.
\end{align*}
The proof for \( a_{(1)} \pi_{M_{\Lambda}(R)}(a_{(2)}) = a \) is similar. This is the desired conclusion.
\end{proof}

\begin{prop}
The maps \( \Delta_{M_{\Lambda}(R)} \) and \( \pi_{M_{\Lambda}(R)} \) satisfy the conditions \eqref{def:st} - \eqref{def:pmulti}.
\end{prop}

\begin{proof}
We first show \eqref{def:st}. For any \( r, r^{\prime} \in R \), \( p, q \in Q^{(m)} \) and \( m \in \mathbb{Z}_{\geq 0} \), we write \( a = r \otimes r^{\prime} \otimes \displaystyle\mathbf{e}\genfrac{[}{]}{0pt}{}{p}{q} + \mathfrak{I}_{\mathbf{w}} \). Let \( f \) be an arbitrary element in \( M_{\Lambda}(R) \). We can evaluate that
\begin{align*}
&a_{(1)} t_{M_{\Lambda}(R)}( f ) \otimes a_{(2)} \\
=& \sum_{u \in Q^{(m)}} (r \otimes f(\mathfrak{t}(u)) \otimes \genfrac{[}{]}{0pt}{}{p}{u} + \mathfrak{I}_{\mathbf{w}}) \otimes (1_R \otimes r^{\prime} \otimes \genfrac{[}{]}{0pt}{}{u}{q} + \mathfrak{I}_{\mathbf{w}}) \\
=& \sum_{u \in Q^{(m)}} t_{M_{\Lambda}(R)}( f( \mathfrak{t}(u) )_M ) (r \otimes 1_R \otimes \genfrac{[}{]}{0pt}{}{p}{u} + \mathfrak{I}_{\mathbf{w}}) \otimes (1_R \otimes r^{\prime} \otimes \genfrac{[}{]}{0pt}{}{u}{q} + \mathfrak{I}_{\mathbf{w}}) \\
=& \sum_{u \in Q^{(m)}} (r \otimes 1_R \otimes \genfrac{[}{]}{0pt}{}{p}{u} + \mathfrak{I}_{\mathbf{w}}) \otimes s_{M_{\Lambda}(R)}( f( \mathfrak{t}(u) )_M )  (1_R \otimes r^{\prime} \otimes \genfrac{[}{]}{0pt}{}{u}{q} + \mathfrak{I}_{\mathbf{w}}) \\
=& \sum_{u \in Q^{(m)}} (r \otimes 1_R \otimes \genfrac{[}{]}{0pt}{}{p}{u} + \mathfrak{I}_{\mathbf{w}}) \otimes (f( \mathfrak{t}(u) \otimes r^{\prime} \otimes \genfrac{[}{]}{0pt}{}{u}{q} + \mathfrak{I}_{\mathbf{w}}) \\
=& a_{(1)} \otimes a_{(2)} s_{M_{\Lambda}(R)}(f).
\end{align*}
Therefore \eqref{def:st} is satisfied.

We next prove \eqref{def:Duni}. For any \( \lambda \in \Lambda \), 
\begin{align*}
&(t_{M_{\Lambda}(R)}( \delta_{\lambda} ) \otimes 1_{\mathfrak{A}(w)} - 1_{\mathfrak{A}(w)} \otimes s_{M_{\Lambda}(R)}( \delta_{\lambda} ) ) 
( \sum_{\mu \in \Lambda} 1_R \otimes 1_R \otimes \mathbf{e}\genfrac{[}{]}{0pt}{}{\mu}{\lambda} + \mathfrak{I}_{\mathbf{w}} \otimes 1_{\mathfrak{A}(w)} )\\
=& \sum_{\substack{\mu, \tau, \nu \in \Lambda \\ \lambda \neq \tau}} (1_R \otimes 1_R \otimes \mathbf{e}\genfrac{[}{]}{0pt}{}{\mu}{\lambda} + \mathfrak{I}_{\mathbf{w}}) \otimes (1_R \otimes 1_R \otimes \mathbf{e}\genfrac{[}{]}{0pt}{}{\tau}{\nu} + \mathfrak{I}_{\mathbf{w}}) \in \mathfrak{I}_2.
\end{align*}
Thus we can induce that
\begin{align*}
&\overline{\nabla}( 1_R \otimes 1_R \otimes 1_{\mathfrak{G}(Q)} ) + \mathfrak{I}_2 \\
=& \sum_{\lambda, \mu, \nu \in \Lambda} (1_R \otimes 1_R \otimes \mathbf{e}\genfrac{[}{]}{0pt}{}{\lambda}{\nu} + \mathfrak{I}_{\mathbf{w}}) \otimes (1_R \otimes 1_R \otimes \mathbf{e}\genfrac{[}{]}{0pt}{}{\nu}{\mu} + \mathfrak{I}_{\mathbf{w}}) \\ 
& + \sum_{\substack{\lambda, \mu, \nu, \tau \in \Lambda \\ \nu \neq \tau}} (1_R \otimes 1_R \otimes \mathbf{e}\genfrac{[}{]}{0pt}{}{\lambda}{\nu} + \mathfrak{I}_{\mathbf{w}}) \otimes (1_R \otimes 1_R \otimes \mathbf{e}\genfrac{[}{]}{0pt}{}{\tau}{\mu} + \mathfrak{I}_{\mathbf{w}}) + \mathfrak{I}_2 \\
=& \sum_{\lambda, \mu, \nu, \tau \in \Lambda} (1_R \otimes 1_R \otimes \mathbf{e}\genfrac{[}{]}{0pt}{}{\lambda}{\nu} + \mathfrak{I}_{\mathbf{w}}) \otimes (1_R \otimes 1_R \otimes \mathbf{e}\genfrac{[}{]}{0pt}{}{\tau}{\mu} + \mathfrak{I}_{\mathbf{w}}) + \mathfrak{I}_2.
\end{align*}
Since \( \mathfrak{A}(w) \otimes_{\mathbb{K}} \mathfrak{A}(w) / \mathfrak{I}_2 \cong  \mathfrak{A}(w) \otimes_{M_{\Lambda}(R)} \mathfrak{A}(w) \) as \( \mathbb{K} \)-vector spaces, \eqref{def:Duni} is proved.

The proof for \eqref{def:Dmulti} is similar to that of multiplicativity of the map \( \overline{\nabla} \).

Let us prove \eqref{def:punit}. Because \( 1_{\mathfrak{G}(Q)} = \sum_{\lambda, \mu \in \Lambda} \displaystyle\mathbf{e}\genfrac{[}{]}{0pt}{}{\lambda}{\mu} \), 

\begin{align*}
\pi_{M_{\Lambda}(R)}( 1_{\mathfrak{A}(w)} ) =& \sum_{\lambda, \mu \in \Lambda} \delta_{\lambda, \mu} \delta_{\mu} \\
=& \sum_{\lambda \in \Lambda} \delta_{\lambda} = 1_{M_{\Lambda}(R)}.
\end{align*}

Finally, we give a proof of \eqref{def:pmulti}. Because \( \zeta \) is a \( \mathbb{K} \)-algebra homomorphism, it is sufficient to prove that \( \zeta(a)( 1_{M_{\Lambda}(R)} ) =\zeta( s_{M_{\Lambda}(R)}( \pi_{M_{\Lambda}(R)}( a ) ) )( 1_{M_{\Lambda}(R)} ) = \zeta( t_{M_{\Lambda}(R)}( \pi_{M_{\Lambda}(R)}( a ) ) )( 1_{M_{\Lambda}(R)} ) \) for all \( a \in \mathfrak{A}(w) \). Let \( r, r^{\prime} \in R \), \( p, q \in Q^{(m)} \), and \( m \in \mathbb{Z}_{\geq 0} \). We can evaluate that
\begin{align*}
&\zeta( s_{M_{\Lambda}(R)}( \pi_{M_{\Lambda}(R)}( r \otimes r^{\prime} \otimes \mathbf{e}\genfrac{[}{]}{0pt}{}{p}{q} + \mathfrak{I}_{\mathbf{w}} ) ) )( 1_{M_{\Lambda}(R)} ) \\
=& \sum_{\lambda \in \Lambda} \delta_{p, q} \zeta( r r^{\prime} \otimes 1_R \otimes \mathbf{e}\genfrac{[}{]}{0pt}{}{\mathfrak{s}(q)}{\lambda} + \mathfrak{I}_{\mathbf{w}} )( 1_{M_{\Lambda}(R)} ) \\
=& \delta_{p, q} (r r^{\prime})_M \delta_{\mathfrak{s}(q)} \\
=& \zeta( r \otimes r^{\prime} \otimes \mathbf{e}\genfrac{[}{]}{0pt}{}{p}{q} + \mathfrak{I}_{\mathbf{w}} )( 1_{M_{\Lambda}(R)} )
\end{align*}
The proof for \( \displaystyle \zeta( t_{M_{\Lambda}(R)}( \pi_{M_{\Lambda}(R)}( r \otimes r^{\prime} \otimes \mathbf{e}\genfrac{[}{]}{0pt}{}{p}{q} + \mathfrak{I}_{\mathbf{w}} ) ) )( 1_{M_{\Lambda}(R)} ) = \zeta( r \otimes r^{\prime} \otimes \mathbf{e}\genfrac{[}{]}{0pt}{}{p}{q} + \mathfrak{I}_{\mathbf{w}} )( 1_{M_{\Lambda}(R)} ) \) is similar. Thus we conclude \eqref{def:pmulti}. This completes the proof.
\end{proof}
The sextuplet \( (\mathfrak{A}(w), M_{\Lambda}(R), s_{M_{\Lambda}(R)}, t_{M_{\Lambda}(R)}, \Delta_{M_{\Lambda}(R)}, \pi_{M_{\Lambda}(R)}) \) is therefore a left bialgebroid by the above propositions.

\section{Left bialgebroid homomorphism \( \Phi \)} \label{sec:Phi}

In this section, we induce a left bialgebroid \( \mathfrak{A}(w_{\sigma}) \) in Subsection \ref{sec:Aw} from the settings of the left bialgebroid \( A_{\sigma} \) in Subsection \ref{sec:As}, and construct a left bialgebroid homomorphism \( \Phi \) from \( \mathfrak{A}(w_{\sigma}) \) to \( A_{\sigma} \). This is a generalization of \cite{matsu}.

Let \( A_{\sigma} \) be a left bialgebroid in Subsection \ref{sec:As} and \( \sigma^{ab}_{cd} \in M_{\Lambda}(R) \; ( a, b, c, d \in X ) \) satisfying the condition \eqref{cond:invc}. We define a quiver \( Q \) over \( \Lambda \) by
\begin{equation}
Q := \Lambda \times X, \; \mathfrak{s}(\lambda, x) = \lambda, \; \mathfrak{t}(\lambda, x) = \lambda \deg(x) \;\; (\lambda \in \Lambda, x \in X) \label{quilx}
\end{equation}
and set
\begin{equation}
\mathbf{w}\begin{sumibmatrix} & (\lambda, a) & \\ (\mu, c) & & (\lambda^{\prime}, b) \\ & (\mu^{\prime}, d) & \end{sumibmatrix} = \delta_{\lambda,  \mu} \sigma^{ba}_{dc}( \lambda ) \label{wsig}
\end{equation}
for all \( ( (\lambda, a), ( \lambda^{\prime}, b ) ), ( (\mu, c), (\mu^{\prime}, d) ) \in Q^{(2)} \). 

\begin{prop}
The definition \eqref{wsig} satisfies the condition \( \eqref{face} \). 
\end{prop}

\begin{proof}
Let \( ( (\lambda, a), ( \lambda^{\prime}, b ) ), ( (\mu, c), (\mu^{\prime}, d) ) \in Q^{(2)} \). \( \mathbf{w}\begin{sumibmatrix} & (\lambda, a) & \\ (\mu, c) & & (\lambda^{\prime}, b) \\ & (\mu^{\prime}, d) & \end{sumibmatrix} \in Z(R) \) is clear because of \( \sigma^{ba}_{dc}( \lambda ) \in Z(R) \).

We next prove that \( \mathbf{w}\begin{sumibmatrix} & (\lambda, a) & \\ (\mu, c) & & (\lambda^{\prime}, b) \\ & (\mu^{\prime}, d) & \end{sumibmatrix} = 0 \) if \( \mathfrak{s}(\lambda, a) \neq \mathfrak{s}(\mu, c) \) or \( \mathfrak{t}(\lambda^{\prime}, b) \neq \mathfrak{t}(\mu^{\prime}, d) \). It follows from \eqref{wsig} that \( \mathbf{w}\begin{sumibmatrix} & (\lambda, a) & \\ (\mu, c) & & (\lambda^{\prime}, b) \\ & (\mu^{\prime}, d) & \end{sumibmatrix} = 0 \) if \( \mathfrak{s}(\lambda, a) \neq \mathfrak{s}(\mu, c) \).

Suppose that \( \mathfrak{t}(\lambda^{\prime}, b) \neq \mathfrak{t}(\mu^{\prime}, d) \). By the definition of the fiber product of the quiver \( Q \), we have
\begin{equation*}
\mathfrak{t}(\lambda^{\prime}, b) = \lambda \deg(a) \deg(b) \;\;\; {\rm and} \;\;\; \mathfrak{t}(\mu^{\prime}, d) = \mu \deg(c) \deg(d).
\end{equation*}
If \( \lambda = \mu \), then \( \mathbf{w}\begin{sumibmatrix} & (\lambda, a) & \\ (\mu, c) & & (\lambda^{\prime}, b) \\ & (\mu^{\prime}, d) & \end{sumibmatrix} = 0 \) because \( \sigma^{ba}_{dc}(\lambda) = 0 \). This completes the proof.
\end{proof}

Therefore we can construct the left bialgebroid \( \mathfrak{A}(w_{\sigma}) := \mathfrak{A}(w) \) in Subsection \ref{sec:Aw}.

\begin{theo} \label{theo:wh}
Let \( r, r^{\prime} \in R \), \( m \in \mathbb{Z}_{\geq 0} \), \( p = ( (\lambda_1, x_1), \ldots, (\lambda_m, x_m) ) \), and \( q = ( (\mu_1, y_1), \ldots, (\mu_m, y_m) ) \in Q^{(m)} \). We define the \( \mathbb{K} \)-linear map \( \overline{\Phi} \colon \mathfrak{H}(Q) \to A_{\sigma} \)  by
\begin{equation*}
\overline{\Phi}( r \otimes r^{\prime} \otimes \mathbf{e}\genfrac{[}{]}{0pt}{}{p}{q} ) = (r_M \otimes r^{\prime}_M) (\delta_{\mathfrak{s}( p )} \otimes \delta_{\mathfrak{s}( q )}) L_{x_1 y_1} \dotsb L_{x_m y_m} + I_{\sigma}.
\end{equation*}
Then \( \overline{\Phi} \) induces a left bialgebroid homomorphism \( ( \Phi \colon \mathfrak{A}(w_{\sigma}) \to A_{\sigma}, {\rm id}_{M_{\Lambda}(R)} ) \).
\end{theo}

\begin{proof}
We first prove that \( \overline{\Phi}( \mathfrak{I}_{w_{\sigma}} ) = \{ 0 \} \). Since the map \( \overline{\Phi} \) is a \( \mathbb{K} \)-algebra homomorphism, we only need to prove that \( \overline{\Phi}( \alpha ) = 0 \) for every generator \( \alpha \) in \eqref{gen:face}. For any \( (( \mu, a ), ( \mu^{\prime}, b )) \), \( (( \nu, c ), ( \nu^{\prime}, d )) \in Q^{(2)} \), 
\begin{align*}
&\sum_{( (\lambda, x), (\lambda^{\prime}, y) ) \in Q^{(2)}} \overline{\Phi}( \mathbf{w}\begin{sumibmatrix} & (\lambda, x) & \\ (\mu, a) & & (\lambda^{\prime}, y) \\ & (\mu^{\prime}, b) & \end{sumibmatrix} \otimes 1_R \otimes \mathbf{e}\genfrac{[}{]}{0pt}{}{(\lambda, x)}{(\nu, c)} \mathbf{e}\genfrac{[}{]}{0pt}{}{(\lambda^{\prime}, y)}{(\nu^{\prime}, d)} ) \\
=& \sum_{\lambda \in \Lambda, x,y \in X} ( \delta_{\lambda, \mu} \sigma^{yx}_{ba}( \lambda )_M \otimes 1_{M_{\Lambda}(R)} ) ( \delta_{\lambda} \otimes \delta_{\nu} ) L_{xc} L_{yd} + I_{\sigma} \\
=& \sum_{x, y \in X} (\sigma^{yx}_{ba} \otimes 1_{M_{\Lambda}(R)})( \delta_{\mu} \otimes \delta_{\nu} ) L_{xc} L_{yd} + I_{\sigma} \\
=& ( \delta_{\mu} \otimes \delta_{\nu} ) \sum_{x, y \in X} (\sigma^{yx}_{ba} \otimes 1_{M_{\Lambda}(R)}) L_{xc} L_{yd} + I_{\sigma}, \displaybreak[0]  \\
&\sum_{( (\lambda, x), (\lambda^{\prime}, y) ) \in Q^{(2)}} \overline{\Phi}( 1_R \otimes \mathbf{w}\begin{sumibmatrix} & (\nu, c) & \\ (\lambda, x) & & (\nu^{\prime}, d) \\ & (\lambda^{\prime}, y) & \end{sumibmatrix} \otimes \mathbf{e}\genfrac{[}{]}{0pt}{}{(\mu, a)}{(\lambda, x)} \mathbf{e}\genfrac{[}{]}{0pt}{}{(\mu^{\prime}, b)}{(\lambda^{\prime}, y)} ) \\
=& \sum_{\lambda \in \Lambda, x, y \in X} (1_{M_{\Lambda}(R)} \otimes \delta_{\nu, \lambda} \sigma^{dc}_{yx}( \nu )_M) ( \delta_{\mu} \otimes \delta_{\lambda} ) L_{ax} L_{by} + I_{\sigma} \\
=& \sum_{x, y \in X} (1_{M_{\Lambda}(R)} \otimes \sigma^{dc}_{yx}) ( \delta_{\mu} \otimes \delta_{\nu} ) L_{ax} L_{by} + I_{\sigma} \\
=& ( \delta_{\mu} \otimes \delta_{\nu} ) \sum_{x, y \in X} (1_{M_{\Lambda}(R)} \otimes \sigma^{dc}_{yx}) L_{ax} L_{by} + I_{\sigma}.
\end{align*}
Hence \( \overline{\Phi} \) induces a \( \mathbb{K} \)-algebra homomorphism \( \Phi \colon \mathfrak{A}(w_{\sigma}) \to A_{\sigma} \).

Let us prove that the pair of \( \mathbb{K} \)-algebra homomorphisms \( (\Phi, \mathrm{id}_{M_{\Lambda}(R)}) \) satisfies \eqref{def:sp}-\eqref{def:Dp}. We can prove \eqref{def:sp} as follows:
\begin{align*}
\Phi \circ s^{\mathfrak{A}(w_{\sigma})}_{M_{\Lambda}(R)} (f) =& \sum_{\lambda, \mu \in \Lambda} ( f(\lambda) \otimes 1_{M_{\Lambda}(R)} )( \delta_{\lambda} \otimes \delta_{\mu} ) + I_{\sigma} \\
=& \sum_{\lambda \in \Lambda} f(\lambda) \delta_{\lambda} \otimes 1_{M_{\Lambda}(R)} + I_{\sigma} \\
=& f \otimes 1_{M_{\Lambda}(R)} + I_{\sigma} = s^{A_{\sigma}}_{M_{\Lambda}(R)}(f).
\end{align*}
The proof of \eqref{def:tp} is similar to that of \eqref{def:sp}. We next prove \eqref{def:piP}. 
Since \( T_{\deg(a)} \) is a \( \mathbb{K} \)-algebra homomorphism for all \( a \in X \), the left hand side of \eqref{def:piP} satisfies that
\begin{align*}
&\pi^{A_{\sigma}}_{M_{\Lambda}(R)} \circ \Phi( r \otimes r^{\prime} \otimes \mathbf{e}\genfrac{[}{]}{0pt}{}{p}{q} + \mathfrak{I}_{\mathbf{w}} ) \\
=& \chi( (r_M \otimes r^{\prime}_M) (\delta_{\lambda_1} \otimes \delta_{\mu_1}) L_{x_1y_1} \dotsb L_{x_m y_m} + I_{\sigma} )( 1_{M_{\Lambda}(R)} ) \\
=& \delta_{x_1, y_1} \dotsb \delta_{x_m, y_m} (r r^{\prime})_M \delta_{\lambda_1} \delta_{\mu_1}.
\end{align*}
For any \( i \in \{ 1, \dotsb , m-1 \} \), we can induce that \( \lambda_{i+1} = \lambda_{i} \deg(x_i) \) and \( \mu_{i+1} = \mu_i \deg( y_i ) \). This fact implies that
\begin{align*}
&\pi^{\mathfrak{A}(w_{\sigma})}_{M_{\Lambda}(R)} ( r \otimes r^{\prime} \otimes \mathbf{e}\genfrac{[}{]}{0pt}{}{p}{q} + \mathfrak{I}_{\mathbf{w}} ) \\
=& \delta_{\lambda_1, \mu_1} \delta_{x_1, y_1} \dotsb \delta_{x_m, y_m} (r r^{\prime})_M \delta_{\mu_1} \\
=& 
\begin{cases}
(rr^{\prime})_M \delta_{\lambda_1}, & \text{\( ( p = q ); \)} \\
0, & \text{(otherwise).}
\end{cases}
\end{align*}
We conclude \eqref{def:piP} because of the above calculation. Finally, we give a proof of \eqref{def:Dp}.
\begin{align*}
&\Delta^{A_{\sigma}}_{M_{\Lambda}(R)} \circ \Phi ( r \otimes r^{\prime} \otimes \mathbf{e}\genfrac{[}{]}{0pt}{}{p}{q} + \mathfrak{I}_{\mathbf{w}} ) \\
=& \sum_{z_1, \ldots, z_m \in X} (r_M \delta_{\lambda_1} \otimes 1_{M_{\Lambda}(R)} ) L_{x_1 z_1} \dotsb L_{x_m z_m} + I_{\sigma} \otimes (1_{M_{\Lambda}(R)} \otimes r^{\prime}_M \delta_{\mu_1}) L_{z_1 y_1} \dotsb L_{z_m y_m} + I_{\sigma}, \\
&(\Phi \otimes \Phi) \circ \Delta^{\mathfrak{A}(w_{\sigma})}_{M_{\Lambda}(R)} ( r \otimes r^{\prime} \otimes \mathbf{e}\genfrac{[}{]}{0pt}{}{p}{q} + \mathfrak{I}_{\mathbf{w}} ) \\
=& \sum_{\substack{\tau \in \Lambda \\ z_1, \ldots, z_m \in X}} (r_M \delta_{\lambda_1} \otimes \delta_{\tau} ) L_{x_1 z_1} \dotsb L_{x_m z_m} + I_{\sigma} \otimes (\delta_{\tau} \otimes r^{\prime}_M \delta_{\mu_1}) L_{z_1 y_1} \dotsb L_{z_m y_m} + I_{\sigma}.
\end{align*}
For any \( \lambda \in \Lambda \), 
\begin{align*}
&(t^{A_{\sigma}}_{M_{\Lambda}(R)}( \delta_{\lambda} ) \otimes 1_{A_{\sigma}} - 1_{A_{\sigma}} \otimes s^{A_{\sigma}}_{M_{\Lambda}(R)}( \delta_{\lambda} ) ) 
( (1_{M_{\Lambda}(R)} \otimes \delta_{\lambda}) + I_{\sigma} \otimes 1_{A_{\sigma}} )\\
=& \sum_{\substack{\mu \in \Lambda \\ \lambda \neq \mu}} (1_{M_{\Lambda}(R)} \otimes \delta_{\lambda}) + I_{\sigma} \otimes ( \delta_{\mu} \otimes 1_{M_{\Lambda}(R)} ) + I_{\sigma} \in I_2.
\end{align*}
Thus \eqref{def:Dp} is proved.
\end{proof}

\begin{ex} \label{ex:ndy}
Let \( \Lambda := \mathbb{Z} / 2 \mathbb{Z} \) and \( X := \mathbb{Z}/2\mathbb{Z} \). For \( a \in \mathbb{Z} / 2 \mathbb{Z} \), \( \deg(a)(\lambda) = a + \lambda \;\; (\lambda \in \Lambda = \mathbb{Z}/2\mathbb{Z}) \).
The map \( \sigma_i : \Lambda \times X \times X \to X \times X \; (i = 1, 2) \) is defined by the following table.
\begin{table}[h]
\begin{center}
\begin{tabular}{l|rr} \hline
\( (\lambda, a, b) \) & \( \sigma_1(\lambda, a, b) \) & \( \sigma_2(\lambda, a, b) \) \\ \hline
\( (0, 0, 0) \) & \( (0, 0) \) & \( (1, 1) \) \\
\( (0, 0, 1) \) & \( (0, 1) \) & \( (1, 0) \) \\
\( (0, 1, 0) \) & \( (0, 1) \) & \( (1, 0) \) \\
\( (0, 1, 1) \) & \( (0, 0) \) & \( (1, 1) \) \\ \hline
\end{tabular}
\;\;
\begin{tabular}{l|rr} \hline
\( (\lambda, a, b) \) & \( \sigma_1(\lambda, a, b) \) & \( \sigma_2(\lambda, a, b) \) \\ \hline
\( (1, 0, 0) \) & \( (1, 1) \) & \( (0, 0) \) \\
\( (1, 0, 1) \) & \( (1, 0) \) & \( (0, 1) \) \\
\( (1, 1, 0) \) & \( (1, 0) \) & \( (0, 1) \) \\
\( (1, 1, 1) \) & \( (1, 1) \) & \( (0, 0) \) \\ \hline
\end{tabular}
\end{center}
\caption{The definition of \( \sigma_i \)}
\end{table}

\noindent
We denote by \( R \) a \( \mathbb{K} \)-algebra. The map \( \sigma^{ab}_{cd} \in M_{\Lambda}(R) \) is defined by
\begin{equation*}
\sigma^{ab}_{cd}( \lambda ) =
\begin{cases}
1_R, & \text{\( ( \sigma_i( \lambda, a, b ) = ( c, d ) ) \)}; \\
0, & \text{(otherwise)}.
\end{cases}
\end{equation*}
The maps \( \deg \) and \( \sigma^{ab}_{cd} \; (a,b,c,d \in X) \) satisfy the condition \eqref{cond:invc}. Thus we can construct the left bialgebroids \( A_{\sigma} \), \( \mathfrak{A}(w_{\sigma}) \), and the left bialgebroid homomorphism \( \Phi \).
\end{ex}

\begin{rem}
The family \( \sigma = \{ \sigma^{ab}_{cd} \}_{a, b, c, d \in X} \) in Example \ref{ex:ndy} is rigid. Thus, this \( A_{\sigma} \) is a Hopf algebroid whose antipode \( S : A_{\sigma} \to A_{\sigma} \) satisfies \( S((L^{-1})_{ab} + I_{\sigma}) = L_{ab} + I_{\sigma} \; (a, b \in X) \).
\end{rem}

\section{Properties of \( \Phi \)} \label{sec:pro}

In this section, we show that \( \mathfrak{A}(w) \), \( A_{\sigma} \), and \( \Phi \) satisfy a certain universal property in case the base algebra \( R \) is a Frobenius-separable \( \mathbb{K} \)-algebra. In order to complete this purpose, we characterize weak bialgebras (weak Hopf algebras) by generalizing the notion of the antipode \( S^{\mathrm{WHA}} \) and Hayashi's antipode \( f^- \) in \cite[Section 2]{hayas}.

We first recall the convolution product. For an arbitrary \( \mathbb{K} \)-coalgebra \( (C, \Delta, \varepsilon) \) and \( \mathbb{K} \)-algebra \( A \), the \( \mathbb{K} \)-vector space \( {\rm Hom}_{\mathbb{K}}(C, A) \) becomes a \( \mathbb{K} \)-algebra by the following multiplication:
\begin{align*}
&(f \star g)( c ) = f(c_{(1)}) f(c_{(2)}) \;\; (f, g \in {\rm Hom}_{\mathbb{K}}(C, A), c \in C); \\
&1_{{\rm Hom}_{\mathbb{K}}(C, A)}(c) = \varepsilon( c ) 1_A.
\end{align*}
This multiplication is called the convolution product.

Let \( A \) be a \( \mathbb{K} \)-algebra and \( e^+ \), \( e^- \), and \( x^+ \) elements in \( A \). An element \( x^{-} \in A \) is called an \( (e^+, e^-) \)-generalized inverse of \( x^+ \) if the following conditions are satisfied:
\begin{equation*}
x^{\pm} x^{\mp} = e^{\mp}, \;\; x^{\pm} x^{\mp}  x^{\pm} =  x^{\pm}.
\end{equation*}
We can easily check that the \( (e^+, e^-) \)-generalized inverse of \( x^+ \) is unique if it exists.

\begin{defi}
Let \( H \) be a weak bialgebra, \( A \) a \( \mathbb{K} \)-algebra, and \( f^{+} \colon H \to A \) a \( \mathbb{K} \)-algebra homomorphism. A \( \mathbb{K} \)-linear map \( f^- \colon H \to A \) is called an antipode of \( f^+ \) if \( f^- \) is the \( (f^+ \circ \varepsilon_s, f^+ \circ \varepsilon_t) \)-generalized inverse of \( f^+ \) with regard to the convolution product of \( {\rm Hom}_{\mathbb{K}}(H, A) \).
\end{defi}

The following lemmas are generalizations of \cite[Lemma 2.1 and 2.2]{hayas}.

\begin{lemm}
Let \( H \) be a weak bialgebra. 
\begin{enumerate}

  \item This \( H \) is a weak Hopf algebra with the antipode \( S \) if and only if \( S \in {\rm End}_{\mathbb{K}}( H )  \) is the antipode of \( {\rm id}_H \).
  \item If \( H^{\prime} \) is a weak Hopf algebra with the antipode \( S \) and \( f^+ \colon H \to H^{\prime} \) is a weak bialgebra homomorphism, then \( S \circ f^+ \) is the antipode of \( f^{+} \).
\end{enumerate}
\end{lemm}
\begin{proof}
We first prove 1. It is clear that \( H \) becomes a weak Hopf algebra whose antipode is \( S \) if \( S \in {\rm End}_{\mathbb{K}}( H ) \) is the antipode of \( {\rm id}_H \). Suppose that \( H \) is a weak Hopf algebra with the antipode \( S \). We give the proof only for \( {\rm id}_H \star S \star {\rm id}_H = {\rm id}_H \). By using \eqref{def:wdm}, 
\begin{align*}
h_{(1)} S( h_{(2)} ) h_{(3)} &= h_{(1)} \varepsilon_s( h_{(2)} ) \\
&= h_{(1)} 1_{(1)} \varepsilon( h_{(2)} 1_{(2)} ) \\
&= h
\end{align*}
for any \( h \in H \). Hence the antipode of \( {\rm id}_H \) is \( S \in {\rm End}_{\mathbb{K}}( H ) \).

Let us show 2. Since \( f^+ \) is a weak bialgebra homomorphism, 
\begin{align*}
\varepsilon_t \circ f^+( h ) &= \varepsilon_{H^{\prime}}( 1_{(1)} f^+(h) ) 1_{(2)} \\
&= \varepsilon_{H^{\prime}}( f^+( 1_{(1)} h ) ) f^+( 1_{(2)} ) \\
&= \varepsilon_H( 1_{(1)} h ) f^+( 1_{(2)} ) \\
&= f^+ \circ \varepsilon_t( h ), \\
( f^+ \star S \circ f^+ )( h ) &= f^+(h)_{(1)} S( f^+(h)_{(2)} ) \\
&= \varepsilon_t \circ f^+( h ) \\
&= f^+ \circ \varepsilon_t( h ).
\end{align*}
Similarly, we can also prove that \( \varepsilon_s \circ f^+ = f^+ \circ \varepsilon_s \) and \( S \circ f^+ \star f^+ = f^+ \circ \varepsilon_s \). The identity \eqref{def:wdm} induce that
\begin{align*}
( f^+ \star S \circ f^+ \star f^+ )( h ) &= f^+( h )_{(1)} \varepsilon_s ( f^+( h )_{(2)} ) \\
&= f^+( h )_{(1)} 1_{(1)} \varepsilon_{H^{\prime}}( f^+( h )_{(2)} 1_{(2)} ) \\
&= f^+( h )
\end{align*}
for all \( h \in H \). The proof for \( S \circ f^+ \star f^+ \star S \circ f^+ = S \circ f^+ \) is similar.
\end{proof}

\begin{lemm} \label{lem:f-hom}
Let \( H \) be a weak bialgebra, \( A \) a \( \mathbb{K} \)-algebra, and \( f^+ \colon H \to A \) a \( \mathbb{K} \)-algebra homomorphism.
\begin{enumerate}

  \item If \( f^+ \) has the antipode \( f^- \), then \( f^- \colon H \to A^{op} \) is a \( \mathbb{K} \)-algebra homomorphism.
  \item In addition to the above situation 1, if \( A \) is a weak bialgebra and \( f^+ \) is a weak bialgebra homomorphism, then the antipode \( f^- \colon H \to A^{bop} \) is a weak bialgebra homomorphism. 
\end{enumerate}
\end{lemm}
\begin{proof}
Let us first show 1. For \( g, h \in H \), 
\begin{align*}
f^-(gh) &= f^-( g_{(1)} h_{(1)} ) f^+ \circ \varepsilon_t( g_{(2)} h_{(2)} ) \\
&\underset{\eqref{lem:est2}}= f^-( g_{(1)} h_{(1)} ) f^+ \circ \varepsilon_t( g_{(2)} \varepsilon_t( h_{(2)} ) ) \\
&\underset{\eqref{lem:esco}}= f^-( g_{(1)} h_{(1)} ) f^+ \circ \varepsilon_t( \varepsilon_t( g_{(2)} h_{(2)} ) g_{(3)} ) \\
&\underset{\eqref{lem:estm}}= f^-( g_{(1)} h_{(1)} ) f^+( \varepsilon_t( g_{(2)} h_{(2)} ) \varepsilon_t( g_{(3)} ) ) \\
&= f^-( g_{(1)} h_{(1)} ) f^+( \varepsilon_t( g_{(2)} h_{(2)} ) g_{(3)} ) f^-(g_{(4)}) \\
&\underset{\eqref{lem:esco}}= f^-( g_{(1)} h_{(1)} ) f^+( g_{(2)} \varepsilon_t( h_{(2)} ) ) f^-( g_{(3)} ) \\
&= f^-( g_{(1)} h_{(1)} ) f^+( g_{(2)} ) f^+( h_{(2)} ) f^-( h_{(3)} ) f^-( g_{(3)} ) \\
&= f^+ \circ \varepsilon_s ( g_{(1)} h_{(1)} ) f^-( h_{(2)} ) f^-( g_{(2)} ) \\
&\underset{\eqref{lem:est2}}= f^+ \circ \varepsilon_s ( \varepsilon_s( g_{(1)} ) h_{(1)} ) f^-( h_{(2)} ) f^-( g_{(2)} ) \\
&= f^-( \varepsilon_s( g_{(1)} )_{(1)} h_{(1)} ) f^+( \varepsilon_s( g_{(1)} )_{(2)} h_{(2)} ) f^-( h_{(3)} ) f^-( g_{(2)} ) \\
&\underset{\eqref{lem:coes}}= f^-( 1_{(1)} h_{(1)} ) f^+( \varepsilon_s( g_{(1)} ) 1_{(2)} h_{(2)} ) f^-( h_{(3)} ) f^-( g_{(2)} ) \\
&\underset{\eqref{def:wdm}}= f^-( h_{(1)} ) f^+( \varepsilon_s( g_{(1)} ) h_{(2)} ) f^-( h_{(3)} ) f^-( g_{(2)} ) \\
&= f^-( h_{(1)} ) f^+ \circ \varepsilon_s( g_{(1)} ) f^+ \circ \varepsilon_t( h_{(2)} ) f^-( g_{(2)} ) \\
&\underset{\eqref{lem:estcomm}}= f^-( h_{(1)} )  f^+ \circ \varepsilon_t( h_{(2)} )  f^+ \circ \varepsilon_s( g_{(1)} ) f^-( g_{(2)} ) \\
&= f^-( h ) f^-( g ).
\end{align*}
In addition, this \( f^- \) preserves the unit. By using \eqref{def:D13} and \eqref{lem:estu},
\begin{align*}
f^-(1_H) &= f^-( 1_{(1)} ) f^+( 1_{(2)} 1^{\prime}_{(1)} ) f^-( 1^{\prime}_{(2)} ) \\
&= f^+ \circ \varepsilon_s( 1_H ) f^+ \circ \varepsilon_t( 1_H ) \\
&= 1_{A}.
\end{align*}
Therefore \( f^- \colon H \to A^{op} \) is a \( \mathbb{K} \)-algebra homomorphism. 

We next prove 2. In order to complete this purpose, we assume the following lemma for the moment (cf.~\cite[Lemma B1 and B2]{nil}).
\begin{lemm} \label{lem:whip}
Let \( H \) and \( A \) be a weak bialgebra. If a weak bialgebra homomorphism \( f^+ \colon H \to A \) has the antipode \( f^- \), the following conditions are satisfied for all \( g, h \in H \):
\begin{align}
&g h_{(1)} \otimes f^-( h_{(2)} ) f^+( h_{(3)} ) = g_{(1)} h_{(1)} \otimes f^-( g_{(2)} h_{(2)} ) f^+( g_{(3)} ) f^+( h_{(3)} ); \label{lem:gh-+} \\
&h_{(1)} g \otimes f^+( h_{(2)} ) f^-( h_{(3)} ) = h_{(1)} g_{(1)} \otimes f^+( h_{(2)} ) f^+( g_{(2)} ) f^-( h_{(3)} g_{(3)} ); \label{lem:hg+-} \\
&f^-( h_{(1)} ) \otimes f^-( h_{(2)} ) = f^-( h_{(1)} ) f^+( h_{(4)} ) f^-( h_{(5)} ) \otimes f^-( h_{(2)} ) f^+( h_{(3)} ) f^-( h_{(6)} ). \label{lem:145}
\end{align}
\end{lemm}
\noindent
For the comultiplicativity of \( f^- \), it is equivarent to show that
\begin{equation*}
f^-( h_{(2)} ) \otimes f^-( h_{(1)} ) = f^-( h )_{(1)} \otimes f^-( h )_{(2)}
\end{equation*}
for all \( h \in H \). We set
\begin{align*}
J = ( f^-(1_{(2)}) \otimes f^-(1_{(1)}) ) \Delta( f^+ \circ \varepsilon_t( 1_{(3)} ) ), \\
\tilde{J} = \Delta( f^+ \circ \varepsilon_s( 1_{(1)} ) )( f^-(1_{(3)}) \otimes f^-(1_{(2)}) ).
\end{align*}
For \( J \) and \( \tilde{J} \), see \cite[Proposition B4]{nil}. Let \( h \) be an arbitrary element in \( H \). 
\begin{align*}
&(f^-( h_{(2)} ) \otimes f^-( h_{(1)} )) J \\
=& (f^-(1_{(2)} h_{(2)}) \otimes f^-(1_{(1)} h_{(1)})) \Delta( f^+( 1_{(3)} ) f^-( 1_{(4)} ) ) \\
\underset{\eqref{lem:hg+-}}=& (f^-(1_{(2)} h_{(2)}) \otimes f^-(1_{(1)} h_{(1)})) \Delta( f^+( 1_{(3)} h_{(3)} ) f^-( 1_{(4)} h_{(4)} ) ) \\
\underset{\eqref{def:wdm}}=& (f^-( h_{(2)} ) \otimes f^-( h_{(1)} )) \Delta( f^+( h_{(3)} ) f^-( h_{(4)} ) ) \\
=& (f^+ \circ \varepsilon_s( h_{(2)} ) \otimes f^-( h_{(1)} ) f^+( h_{(3)} )) \Delta( f^-( h_{(4)} ) ) \\
\underset{\eqref{lem:estd2}}=& (f^+ \circ \varepsilon_s( 1_{(2)} ) \otimes f^-( h_{(1)} 1_{(1)} ) f^+( h_{(2)} )) \Delta( f^-( h_{(3)} ) ) \\
=& (f^+ \circ \varepsilon_s( 1_{(2)} ) \otimes f^-( 1_{(1)} ) f^+ \circ \varepsilon_s( h_{(1)} )) \Delta( f^-( h_{(2)} ) ) \\
\underset{\eqref{lem:estd1}}=& (f^+ \circ \varepsilon_s( 1_{(2)} ) \otimes f^-( 1_{(1)} ) f^+ \circ \varepsilon_s( 1^{\prime}_{(1)} )) \Delta( f^-( h 1^{\prime}_{(2)} ) ) \\
=& ( f^-( 1_{(2)} ) f^+( 1_{(3)} ) \otimes f^-( 1_{(1)}^{\prime} 1_{(1)} ) f^+( 1_{(2)}^{\prime} ) ) \Delta( f^-( h 1^{\prime}_{(3)} ) ) \\
\underset{\eqref{lem:gh-+}}=& (f^-(1_{(2)} 1^{\prime}_{(2)}) f^+(1_{(3)} 1^{\prime}_{(3)}) \otimes f^-(1_{(1)} 1^{\prime}_{(1)}) f^+(1_{(4)})) \Delta( f^-( h 1_{(5)} ) ) \\
\underset{\eqref{def:wdm}}=& (f^-(1_{(2)}) \otimes f^-(1_{(1)})) \Delta( f^+(1_{(3)}) ) \Delta( f^-(1_{(4)}) ) \Delta( f^-( h ) ) \\
=& J ( f^-(h)_{(1)} \otimes f^-(h)_{(2)} ).
\end{align*}
By using \eqref{def:D13}, \eqref{lem:estu}, and \eqref{lem:estd1}, \( J \) satisfies that
\begin{align*}
J &= (f^+ \circ \varepsilon_s( 1_{(2)} ) \otimes f^-( 1_{(1)} ) f^+(1_{(3)})) \Delta( f^-(1_{(4)}) ) \\
&= (f^+( 1_{(1)} ) \otimes f^-( 1^{\prime}_{(1)} ) f^+( 1^{\prime}_{(2)} 1_{(2)} )) \Delta(f^-(1^{\prime}_{(3)})) \\
&= (f^+( 1_{(1)} ) \otimes f^+ \circ \varepsilon_s( 1^{\prime}_{(1)} ) f^+( 1_{(2)} )) \Delta(f^-(1^{\prime}_{(2)})) \\
&= (f^+( 1_{(1)} ) \otimes f^+( 1^{\prime}_{(1)} 1_{(2)} )) \Delta(f^-(1^{\prime}_{(2)})) \\
&= (f^+( 1_{(1)} ) \otimes f^+( 1_{(2)} )) \Delta(f^-(1_{(3)})) \\
&= \Delta( f^+ \circ \varepsilon_t( 1_H ) ) \\
&= \Delta( 1_A ).
\end{align*}
Similarly, we can prove that \( \tilde{J} = \Delta( 1_A ) \). The identities \eqref{def:wdm}, \eqref{lem:hg+-}, and \eqref{lem:145} induce that
\begin{align*}
J \tilde{J} &= J \Delta( f^-( 1_{(1)} ) ) \Delta( f^+( 1_{(2)} ) ) ( f^-(1_{(4)}) \otimes f^-(1_{(3)}) ) \\
&= (f^-(1_{(2)}) \otimes f^-(1_{(1)})) J \Delta( f^+( 1_{(3)} ) ) ( f^-(1_{(5)}) \otimes f^-(1_{(4)}) ) \\
&= (f^-(1_{(2)} 1^{\prime}_{(2)}) \otimes f^-( 1_{(1)} 1^{\prime}_{(1)} )) \Delta( f^+(1_{(3)}) f^-(1_{(4)}) f^+(1^{\prime}_{(3)}) ) (f^-( 1^{\prime}_{(5)} ) \otimes f^-( 1^{\prime}_{(4)} )) \\
&= (f^-(1_{(2)} 1^{\prime}_{(2)}) \otimes f^-( 1_{(1)} 1^{\prime}_{(1)} )) \\
& \;\;\; \times \Delta( f^+(1_{(3)} 1^{\prime}_{(3)}) f^-(1_{(4)} 1^{\prime}_{(4)}) f^+(1^{\prime}_{(5)}) ) (f^-( 1^{\prime}_{(7)} ) \otimes f^-( 1^{\prime}_{(6)} )) \\
&= (f^-(1_{(2)}) \otimes f^-( 1_{(1)} )) \Delta( f^+(1_{(3)}) ) (f^-( 1_{(5)} ) \otimes f^-( 1_{(4)} )) \\
&= f^-( 1_{(2)} ) f^+( 1_{(3)} ) f^-( 1_{(6)} ) \otimes f^-( 1_{(1)} ) f^+( 1_{(4)} ) f^-( 1_{(5)} ) \\
&= f^-(1_{(2)}) \otimes f^-(1_{(1)}).
\end{align*}
We can calculate that
\begin{align*}
( f^-(h)_{(1)} \otimes f^-(h)_{(2)} ) &= \Delta(1_A)( f^-(h)_{(1)} \otimes f^-(h)_{(2)} ) \\
&= J ( f^-(h)_{(1)} \otimes f^-(h)_{(2)} ), \\
&= (f^-(h_{(2)}) \otimes f^-(h_{(1)}))J \\ 
&= (f^-(h_{(2)}) \otimes f^-(h_{(1)})) \Delta( 1_A ) \\
&= (f^-(h_{(2)}) \otimes f^-(h_{(1)})) J \tilde{J} \\
&= (f^-(h_{(2)}) \otimes f^-(h_{(1)})) (f^-(1_{(2)}) \otimes f^-(1_{(1)})) \\
&= (f^-(h_{(2)}) \otimes f^-(h_{(1)}))
\end{align*}
for any \( h \in H \). Thus \( f^- \colon H \to A^{bop} \) preserves the comultiplication.
By using \eqref{lem:eest2}, we can prove that \( f^- \) is counital:
\begin{align*}
\varepsilon_A \circ f^-( h ) &= \varepsilon_A( f^-( h_{(1)} ) f^+ \circ \varepsilon_t( h_{(2)} ) ) \\
&= \varepsilon_A( f^-( h_{(1)} ) f^+( h_{(2)} ) ) \\
&= \varepsilon_A( \varepsilon_s \circ f^+( h ) ) \\
&= \varepsilon_A \circ f^+( h ) \\
&= \varepsilon_H( h )
\end{align*}
for \( h \in H \). This is the desired conclusion. 
\end{proof}

\begin{proof}[{\bf Proof of Lemma \ref{lem:whip}}]
We first prove \eqref{lem:gh-+}. For all \( g, h \in H \),
\begin{align*}
g h_{(1)} \otimes f^-( h_{(2)} ) f^+( h_{(3)} ) &= g h_{(1)} \otimes f^+ \circ \varepsilon_s( h_{(2)} ) \\
&= g h 1_{(1)} \otimes f^+ \circ \varepsilon_s( 1_{(2)} ) \\
&= g_{(1)} h_{(1)} \otimes f^+ \circ \varepsilon_s( g_{(2)} h_{(2)} ) \\
&= g_{(1)} h_{(1)} \otimes f^-( g_{(2)} h_{(2)} ) f^+( g_{(3)} ) f^+( h_{(3)} ).
\end{align*}
Here we use the identity \eqref{lem:estd2}. The proof for \eqref{lem:hg+-} is similar. 

Let us evaluate \eqref{lem:145}. By using \eqref{lem:gh-+}, \eqref{lem:hg+-}, and Lemma \ref{lem:f-hom}-1, 
\begin{align*}
&f^-( h_{(1)} ) \otimes f^-( h_{(2)} ) \\
=& f^-(h_{(1)}) f^-(1_{(1)}) f^+(1_{(2)}) f^-(1_{(3)}) \otimes f^-(h_{(2)}) f^+(h_{(3)}) f^-(h_{(4)}) \\
=& f^-(1_{(1)} h_{(1)}) f^+(1_{(4)}) f^-(1_{(5)}) \otimes f^-(1_{(2)} h_{(2)}) f^+(1_{(3)} h_{(3)}) f^-(h_{(4)}) \\
=& f^-(1_{(1)} h_{(1)}) f^+(1_{(4)} h_{(4)}) f^-(1_{(5)} h_{(5)}) \otimes f^-(1_{(2)} h_{(2)}) f^+(1_{(3)} h_{(3)}) f^-(h_{(6)}) \\
=& f^-( h_{(1)} ) f^+( h_{(4)} ) f^-( h_{(5)} ) \otimes f^-( h_{(2)} ) f^+( h_{(3)} ) f^-( h_{(6)} )
\end{align*}
for any \( h \in H \). This completes the proof. 
\end{proof}

The convolution product and the antipode \( f^- \) generalize the notion of the Hopf envelope in \cite{benp}. 

\begin{defi} \label{def:WHC}
Let \( H \) be a weak bialgebra, \( \overline{H} \) a weak Hopf algebra, and \( \iota \colon H \to \overline{H} \) a weak bialgebra homomorphism. A Hopf closure of \( H \) is a pair \( (\overline{H}, \iota) \) satisfying the following universal property:
\begin{itemize}
\item[]
For any weak bialgebra \( B \) and any weak bialgebra homomorphism \( f^+ \colon H \to B \) with the antipode \( f^- \), there exists a unique weak bialgebra homomorphism \( F \colon \overline{H} \to B \) such that the following diagram is commutative:
\end{itemize}
\[
\xymatrix@C=30pt@R=30pt{
H \ar[rd]_-{f^+} \ar[r]^-{\iota} & \overline{H} \ar[d]^-{F} \\
 & B. \\
}
\]

We can induce that \( \overline{H} \) is unique up to isomorphism if there exists.
\end{defi}

\begin{rem}
\begin{enumerate}
\item In \cite{benp}, the weak bialgebra \( B \) is always a weak Hopf algebra with the antipode \( S \). Thus Definition \ref{def:WHC} is a generalization of Definition 3.14 in \cite{benp} because \( S \circ f^+ \) gives an antipode of \( f^+ \in {\rm Hom}_{\mathbb{K}}(H, B) \).

\item Let \( H \) be a face algebra. Hayashi \cite{hayas} considered construction of the Hopf closure \( \overline{H} \) if \( H \) is coquasitriangular and closurable. Then this \( \overline{H} \) satisfies Definition \ref{def:WHC} replaced with the notion of face algebras, that is to say, \( f^+ \colon H \to B \) and \( F \colon \overline{H} \to B \) are face algebra homomorphisms (see \cite[Theorem 5.1, 8.2, and 8.3] {hayas}).
\end{enumerate}
\end{rem}

Let \( \Lambda \) be a non empty finite set and \( X \) a finite set. For a left bialgebroid \( A_{\sigma} \) in Subsection \ref{sec:As}, we suppose that the \( \mathbb{K} \)-algebra \( R \) is a Frobenius-separable \( \mathbb{K} \)-algebra with an idempotent Frobenius system \( (\psi, e^{(1)} \otimes e^{(2)}) \). This \( A_{\sigma} \) has a weak bialgebra structure by Proposition \ref{prop:LWF}. The quiver \( Q \) defined by \eqref{quilx} and elements \( \mathbf{w}\begin{sumibmatrix} & (\lambda, a) & \\ (\mu, c) & & (\lambda^{\prime}, b) \\ & (\mu^{\prime}, d) & \end{sumibmatrix} \in R \; ( ( (\lambda, a), ( \lambda^{\prime}, b ) ), ( (\mu, c), (\mu^{\prime}, d) ) \in Q^{(2)} ) \) in \eqref{wsig} give birth to a weak bialgebra \( \mathfrak{A}(w_{\sigma}) \) and its homomorphism \( \Phi \) in Section \ref{sec:Phi}.

\begin{theo}
If \( \sigma \) is rigid, the pair \( (A_{\sigma}, \Phi) \) satisfies the following universal property:
\begin{itemize}
\item[]
For any \( \mathbb{K} \)-algebra \( A \) and any \( \mathbb{K} \)-algebra homomorphism \( f^+ \colon \mathfrak{A}(w_{\sigma}) \to A \) with the antipode \( f^- \), there exists a unique \( \mathbb{K} \)-algebra homomorphism \( F \colon A_{\sigma} \to A \) such that the following diagram is commutative:
\end{itemize}
\[
\xymatrix@C=30pt@R=30pt{
\mathfrak{A}(w_{\sigma}) \ar[rd]_-{f^+} \ar[r]^-{\Phi} & A_{\sigma} \ar[d]^-{F} \\
 & A. \\
}
\]
If this \( \mathbb{K} \)-algebra \( A \) has a weak bialgebra structure \( (A, \Delta, \varepsilon) \) and \( f^+ \) is a weak bialgebra homomorphism, then so is \( F \).
\end{theo}
\begin{proof}
We first show the existence of the \( \mathbb{K} \)-algebra homomorphism \( F \). The \( \mathbb{K} \)-algebra \( \overline{F} \colon \mathbb{K} \langle \Lambda X \rangle \to A \) is defined by
\begin{align*}
&\overline{F}(\xi) = \Upsilon(\xi) \;\; (\xi \in M_{\Lambda}(R) \otimes_{\mathbb{K}} M_{\Lambda}(R)^{op}); \\
&\overline{F}(L_{ab}) = \sum_{\lambda, \mu \in \Lambda} f^+ ( 1_R \otimes 1_R \otimes \mathbf{e}\genfrac{[}{]}{0pt}{}{(\lambda, a)}{(\mu, b)} + \mathfrak{I}_{\mathbf{w}} ) \;\; (a,b \in X); \\
&\overline{F}((L^{-1})_{ab}) =  \sum_{\lambda, \mu \in \Lambda} f^- ( 1_R \otimes 1_R \otimes \mathbf{e}\genfrac{[}{]}{0pt}{}{(\lambda, a)}{(\mu, b)} + \mathfrak{I}_{\mathbf{w}} ).
\end{align*}
Here \( \Upsilon \) is a \( \mathbb{K} \)-algebra homomorphism defined by
\begin{equation*}
\Upsilon \colon M_{\Lambda}(R) \otimes_{\mathbb{K}} M_{\Lambda}(R)^{op} \ni g \otimes h \mapsto \sum_{\lambda, \mu \in \Lambda} f^+( g(\lambda) \otimes h(\mu) \otimes \mathbf{e}\genfrac{[}{]}{0pt}{}{\lambda}{\mu} + \mathfrak{I}_{\mathbf{w}} ) \in A.
\end{equation*}
We prove that \( \overline{F}( I_{\sigma} ) = \{ 0 \} \). It suffices to check that
\begin{equation} \label{F0}
\overline{F}( \alpha ) = 0
\end{equation}
 for every generator \( \alpha \) in \( I_{\sigma} \).
For any \( \alpha \) in the generators (1), the condition \eqref{F0} is obviously satisfied since the map \( \Upsilon \) is a \( \mathbb{K} \)-algebra homomorphism. 
We next prove that the generators (2) satisfy \eqref{F0}. By using \eqref{lem:estu}, 
\begin{align*}
&\overline{F}( \sum_{c \in X} (L^{-1})_{ac} L_{cb} ) \\
=& \sum_{\substack{c \in X \\ \lambda, \mu, \tau, \nu \in \Lambda}} f^-( 1_R \otimes 1_R \otimes \mathbf{e}\genfrac{[}{]}{0pt}{}{(\lambda, a)}{(\mu, c)} + \mathfrak{I}_{\mathbf{w}} ) f^+ \circ \varepsilon_s( 1_{\mathfrak{A}(w_{\sigma})} ) \\
& \times f^+( 1_R \otimes 1_R \otimes \mathbf{e}\genfrac{[}{]}{0pt}{}{(\tau, c)}{(\nu, b)} + \mathfrak{I}_{\mathbf{w}} ) \\
=& \sum_{\substack{c \in X \\ \lambda, \mu, \tau \in \Lambda}} f^-( 1_R \otimes e^{(1)} \otimes \mathbf{e}\genfrac{[}{]}{0pt}{}{(\lambda, a)}{(\mu, c)} + \mathfrak{I}_{\mathbf{w}} )  f^+( e^{(2)} \otimes 1_R \otimes \mathbf{e}\genfrac{[}{]}{0pt}{}{(\mu, c)}{(\tau, b)} + \mathfrak{I}_{\mathbf{w}} ) \\
=& \sum_{\lambda, \mu \in \Lambda} f^+ \circ \varepsilon_s( 1_R \otimes 1_R \otimes \mathbf{e}\genfrac{[}{]}{0pt}{}{(\lambda, a)}{(\mu, b)} + \mathfrak{I}_{\mathbf{w}} ) \\
=& \delta_{a, b} \sum_{\lambda, \mu, \tau \in \Lambda} \delta_{\lambda, \tau \deg(b)^{-1}} f^+( 1_R \otimes e^{(1)} \psi(e^{(2)}) \otimes \mathbf{e}\genfrac{[}{]}{0pt}{}{\mu}{\lambda \deg( a )} + \mathfrak{I}_{\mathbf{w}} ) \\
=& \delta_{a, b} \sum_{\lambda, \mu \in \Lambda} f^+(  1_R \otimes 1_R \otimes \mathbf{e}\genfrac{[}{]}{0pt}{}{\lambda}{\mu \deg(b)^{-1} \deg( a )} + \mathfrak{I}_{\mathbf{w}} ) \\
=& \overline{F}( \delta_{a, b} \emptyset )
\end{align*}
for all \( a, b \in X \). Therefore \( \overline{F}( \sum_{c \in X} (L^{-1})_{ac} L_{cb} - \delta_{a, b} \emptyset ) = 0 \) is satisfied for any \( a, b \in X \). We can prove that \(  \overline{F}( \sum_{c \in X} L_{ac} (L^{-1})_{cb} - \delta_{a, b} \emptyset ) = 0 \) for all \( a, b \in X \) by the similar way.
Let us check that any generator \( \alpha \) in \( (3) \) satisfies \eqref{F0}. For \( g \in M_{\Lambda}( R ) \), \( a \), and \( b \in X \),
\begin{align*}
& \overline{F}( ( T_{\deg(a)}( g ) \otimes 1_{M_{\Lambda}(R)} ) L_{ab} - L_{ab} (g \otimes 1_{M_{\Lambda}(R)}) ) \\
=& \sum_{\lambda, \mu, \tau, \nu \in \Lambda} f^+( (g( \lambda \deg(a) ) \otimes 1_R \otimes \mathbf{e}\genfrac{[}{]}{0pt}{}{\lambda}{\mu})(1_R \otimes 1_R \otimes \mathbf{e}\genfrac{[}{]}{0pt}{}{(\tau, a)}{(\nu, b)}) + \mathfrak{I}_{\mathbf{w}} ) \\
-& \sum_{\gamma, \eta, \theta, \kappa \in \Lambda} f^+( (1_R \otimes 1_R \otimes \mathbf{e}\genfrac{[}{]}{0pt}{}{(\gamma, a)}{(\eta, b)})(g( \theta ) \otimes 1_R \otimes \mathbf{e}\genfrac{[}{]}{0pt}{}{\theta}{\kappa}) + \mathfrak{I}_{\mathbf{w}} ) \\
=& \sum_{\lambda, \mu \in \Lambda} f^+( g( \lambda \deg(a) ) \otimes 1_R \otimes \mathbf{e}\genfrac{[}{]}{0pt}{}{(\lambda, a)}{(\mu, b)} + \mathfrak{I}_{\mathbf{w}} ) \\
-& \sum_{\eta, \theta \in \Lambda} f^+( g( \eta \deg(a) ) \otimes 1_R \otimes \mathbf{e}\genfrac{[}{]}{0pt}{}{(\eta, a)}{(\theta, b)} + \mathfrak{I}_{\mathbf{w}} ) \\
=& 0.
\end{align*}
The proof of \( \overline{F}( ( 1_{M_{\Lambda}(R)} \otimes T_{\deg(b)}(g) ) L_{ab} - L_{ab} ( 1_{M_{\Lambda}(R)} \otimes g ) ) = 0 \; (\forall a, b \in X) \) is similar. To complete the proof of the other two generators in (3), we assume the lemma below for the moment.
\begin{lemm} \label{lem:estA}
For any \( r \in R \) and \( \lambda \in \Lambda \), 
\begin{align}
\sum_{\mu \in \Lambda} \varepsilon_s( 1_R \otimes r \otimes \mathbf{e}\genfrac{[}{]}{0pt}{}{\mu}{\lambda} + \mathfrak{I}_{\mathbf{w}} ) = \sum_{\mu \in \Lambda} 1_R \otimes r \otimes \mathbf{e}\genfrac{[}{]}{0pt}{}{\mu}{\lambda} + \mathfrak{I}_{\mathbf{w}}; \label{lem:es} \\
\sum_{\mu \in \Lambda} \varepsilon_t( r \otimes 1_R \otimes \mathbf{e}\genfrac{[}{]}{0pt}{}{\lambda}{\mu} + \mathfrak{I}_{\mathbf{w}} ) = \sum_{\mu \in \Lambda} r \otimes 1_R \otimes \mathbf{e}\genfrac{[}{]}{0pt}{}{\lambda}{\mu} + \mathfrak{I}_{\mathbf{w}}. \label{lem:et}
\end{align}
\end{lemm}
\noindent
Let \( g \in M_{\Lambda}(R) \), \( a \), and \( b \in X \). By using \eqref{lem:estu}, \eqref{lem:estcomm}, and \eqref{lem:et},
\begin{align*}
&\overline{F}( ( g \otimes 1_{M_{\Lambda}(R)} ) (L^{-1})_{ab} ) \\
=& \sum_{\lambda, \mu, \tau, \nu \in \Lambda} f^+( g(\lambda) \otimes 1_R \otimes \mathbf{e}\genfrac{[}{]}{0pt}{}{\lambda}{\mu} + \mathfrak{I}_{\mathbf{w}} ) f^-( 1_R \otimes 1_R \otimes \mathbf{e}\genfrac{[}{]}{0pt}{}{(\tau, a)}{(\nu, b)} + \mathfrak{I}_{\mathbf{w}} ) \\
=& \sum_{\substack{c \in X \\ \lambda, \mu, \tau, \nu, \gamma \in \Lambda}} f^+ \circ \varepsilon_t( g(\lambda) \otimes 1_R \otimes \mathbf{e}\genfrac{[}{]}{0pt}{}{\lambda}{\mu} + \mathfrak{I}_{\mathbf{w}} ) f^+ \circ \varepsilon_s( 1_R \otimes e^{(1)} \otimes \mathbf{e}\genfrac{[}{]}{0pt}{}{(\tau, a)}{(\gamma, c)} + \mathfrak{I}_{\mathbf{w}} ) \\
& \times f^-( e^{(2)} \otimes 1_R \otimes \mathbf{e}\genfrac{[}{]}{0pt}{}{(\gamma, c)}{(\nu, b)} + \mathfrak{I}_{\mathbf{w}} ) \\
=& \sum_{\substack{c, d \in X \\ \lambda, \mu, \tau, \nu \in \Lambda}} f^-( 1_R \otimes e^{(1)} \otimes \mathbf{e}\genfrac{[}{]}{0pt}{}{(\lambda, a)}{(\mu, c)} + \mathfrak{I}_{\mathbf{w}} )  \\
& \times f^+( e^{(2)} g( \mu \deg(c) ) \otimes e^{(1)^{\prime \prime}} e^{(1)^{\prime}} \otimes \mathbf{e}\genfrac{[}{]}{0pt}{}{(\mu, c)}{(\tau, d)} + \mathfrak{I}_{\mathbf{w}} ) \\ 
& \times f^-( e^{(2)^{\prime}} e^{(2)^{\prime \prime}} \otimes 1_R \otimes \mathbf{e}\genfrac{[}{]}{0pt}{}{(\tau, d)}{(\nu, b)} + \mathfrak{I}_{\mathbf{w}} ) \\
=& \sum_{\substack{c \in X \\ \lambda, \mu, \tau \in \Lambda}} f^-( 1_R \otimes e^{(1)} \otimes \mathbf{e}\genfrac{[}{]}{0pt}{}{(\lambda, a)}{(\mu, c)} + \mathfrak{I}_{\mathbf{w}} ) \\
& \times  f^+ \circ \varepsilon_t( e^{(2)} g( \mu \deg(c) ) \otimes 1_R \otimes \mathbf{e}\genfrac{[}{]}{0pt}{}{(\mu, c)}{(\tau, b)} + \mathfrak{I}_{\mathbf{w}} ) \\ 
=& \sum_{\substack{c \in X \\ \lambda, \mu, \tau, \nu \in \Lambda}} \delta_{\mu, \tau} \delta_{b, c} f^-( 1_R \otimes e^{(1)} \otimes \mathbf{e}\genfrac{[}{]}{0pt}{}{(\lambda, a)}{(\mu, c)} + \mathfrak{I}_{\mathbf{w}} ) \\
& \times f^+( \psi( e^{(2)} g( \mu \deg(c) ) e^{(1)^{\prime}} ) e^{(2)^{\prime}} \otimes 1_R \otimes \mathbf{e}\genfrac{[}{]}{0pt}{}{\tau}{\nu} + \mathfrak{I}_{\mathbf{w}} ) \\
=& \sum_{\lambda, \mu, \tau \in \Lambda} f^-( 1_R \otimes e^{(1)} \otimes \mathbf{e}\genfrac{[}{]}{0pt}{}{(\lambda, a)}{(\mu, b)} + \mathfrak{I}_{\mathbf{w}} ) f^+( e^{(2)} g( \mu \deg(b) ) \otimes 1_R \otimes \mathbf{e}\genfrac{[}{]}{0pt}{}{\mu}{\tau} + \mathfrak{I}_{\mathbf{w}} ) \\
=& \sum_{\lambda, \mu, \tau, \nu \in \Lambda}  f^-( 1_R \otimes 1_R \otimes \mathbf{e}\genfrac{[}{]}{0pt}{}{(\lambda, a)}{(\mu, b)} + \mathfrak{I}_{\mathbf{w}} ) f^+ \circ \varepsilon_s( 1_{\mathfrak{A}(w_{\sigma})} ) \\
& \times f^+( g( \tau \deg(b) ) \otimes 1_R \otimes \mathbf{e}\genfrac{[}{]}{0pt}{}{\tau}{\nu} + \mathfrak{I}_{\mathbf{w}} ) \\
=& \overline{F}( (L^{-1})_{ab} ( T_{\deg(b)}( g ) \otimes 1_{M_{\Lambda}(R)} ) )
\end{align*}
\( \overline{F}( (1_{M_{\Lambda}(R)}\otimes g)(L^{-1})_{ab} - (L^{-1})_{ab}(1_{M_{\Lambda}(R)}\otimes T_{\deg (a)}(g)) ) = 0 \; (\forall a, b \in X) \) is also induced by using \eqref{lem:estu}, \eqref{lem:estcomm}, and \eqref{lem:es}.
We give a proof of \eqref{F0} for any generator \( \alpha \) in (4). For all \( a \), \( b \), \( c \), and \( d \in X \),
\begin{align*}
&\overline{F}( \sum_{x,y\in X} (\sigma^{xy}_{ac}\otimes1_{M_{\Lambda}(R)})L_{yd}L_{xb} - \sum_{x,y\in X} (1_{M_{\Lambda}(R)}\otimes\sigma^{bd}_{xy})L_{cy}L_{ax} ) \\
=& \sum_{\substack{\lambda. \mu \in \Lambda \\ x, y \in X}} f^+( \sigma^{xy}_{ac}( \lambda ) \otimes 1_R \otimes \mathbf{e}\genfrac{[}{]}{0pt}{}{((\lambda, y), (\lambda \deg(y), x))}{((\mu, d), (\mu \deg(d), b))} + \mathfrak{I}_{\mathbf{w}} ) \\
-& \sum_{\substack{\tau. \nu \in \Lambda \\ x, y \in X}} f^+( 1_R \otimes \sigma^{bd}_{xy}( \nu ) \otimes \mathbf{e}\genfrac{[}{]}{0pt}{}{((\tau, c), (\tau \deg(c), a))}{((\nu, y), (\nu \deg(y), x))} + \mathfrak{I}_{\mathbf{w}} ) \\
=& \sum_{\substack{\lambda. \mu, \eta \in \Lambda \\ x, y \in X}} f^+( \mathbf{w}\begin{sumibmatrix} & (\lambda, y) & \\ (\eta, c) & & (\lambda \deg(y), x) \\ & (\eta \deg(c), a) & \end{sumibmatrix} \otimes 1_R \otimes \mathbf{e}\genfrac{[}{]}{0pt}{}{((\lambda, y), (\lambda \deg(y), x))}{((\mu, d), (\mu \deg(d), b))} + \mathfrak{I}_{\mathbf{w}} ) \\
-& \sum_{\substack{\tau. \nu, \theta \in \Lambda \\ x, y \in X}} f^+( 1_R \otimes \mathbf{w}\begin{sumibmatrix} & (\theta, d) & \\ (\nu, y) & & (\theta \deg(d), b) \\ & (\nu \deg(y), x) & \end{sumibmatrix} \otimes \mathbf{e}\genfrac{[}{]}{0pt}{}{((\tau, c), (\tau \deg(c), a))}{((\nu, y), (\nu \deg(y), x))} + \mathfrak{I}_{\mathbf{w}} ) \\
=& 0.
\end{align*}
Here we use the setting \eqref{wsig} to show the second equality.
We can easily induce that the generator (5) satisfies \eqref{F0} because of \( 1_{\mathfrak{A}(w_{\sigma})} = \displaystyle\sum_{\lambda, \mu \in \Lambda}  1_R \otimes 1_R \otimes \mathbf{e}\genfrac{[}{]}{0pt}{}{\lambda}{\mu} + \mathfrak{I}_{\mathbf{w}} \). Hence the \( \mathbb{K} \)-algebra homomorphism \( F( \alpha + I_{\sigma} ) = \overline{F}( \alpha ) \; ( \alpha \in \mathbb{K} \langle \Lambda X \rangle ) \) is well defined.

We next show that \( f^+ = F \circ \Phi \). Since these three maps \( f^+ \), \( F \) and \( \Phi \) are \( \mathbb{K} \)-algebra homomorphisms, it is sufficient to prove that
\begin{equation*}
f^+( r \otimes r^{\prime} \otimes \mathbf{e}\genfrac{[}{]}{0pt}{}{(\lambda, a)}{(\mu, b)} + \mathfrak{I}_{\mathbf{w}} ) = F \circ \Phi( r \otimes r^{\prime} \otimes \mathbf{e}\genfrac{[}{]}{0pt}{}{(\lambda, a)}{(\mu, b)} + \mathfrak{I}_{\mathbf{w}} )
\end{equation*}
for all \( r, r^{\prime} \in R \), \( (\lambda, a) \), and \( (\mu, b) \in Q \). We can evaluate that
\begin{align*}
&F \circ \Phi( r \otimes r^{\prime} \otimes \mathbf{e}\genfrac{[}{]}{0pt}{}{(\lambda, a)}{(\mu, b)} + \mathfrak{I}_{\mathbf{w}} ) \\
=& f^+( r \otimes r^{\prime} \otimes \mathbf{e}\genfrac{[}{]}{0pt}{}{\lambda}{\mu} + \mathfrak{I}_{\mathbf{w}} ) ( \sum_{\tau, \nu \in \Lambda} f^+( 1_R \otimes 1_R \otimes \mathbf{e}\genfrac{[}{]}{0pt}{}{(\tau, a)}{(\nu, b)} + \mathfrak{I}_{\mathbf{w}} ) ) \\
=& f^+( r \otimes r^{\prime} \otimes \mathbf{e}\genfrac{[}{]}{0pt}{}{(\lambda, a)}{(\mu, b)} + \mathfrak{I}_{\mathbf{w}} ).
\end{align*}
We give a proof of the uniqueness of \( F \). Let \( F^{\prime} \) be a \( \mathbb{K} \)-algebra homomorphism such that \(  f^+ = F^{\prime} \circ \Phi \). We see at once that \( F^{\prime}( g \otimes h + I_{\sigma} ) = F( g \otimes h + I_{\sigma} ) \) and \( F^{\prime}( L_{ab} + I_{\sigma} ) = F( L_{ab} + I_{\sigma} ) \) for all \( g, h \in M_{\Lambda}(R) \), \( a \), and \( b \in X \). We denote by \( S^{{\rm WHA}} \) the antipode of the weak Hopf algebra \( A_{\sigma} \). According to Proposition \ref{prop:WHD}, this \( S^{{\rm WHA}} \) satisfies that \( S^{{\rm WHA}}( L_{ab} + I_{\sigma} ) = (L^{-1})_{ab} + I_{\sigma} \) for any \( a, b \in X \). Therefore we compute that
\begin{align*}
F^{\prime}( (L^{-1})_{ab} + I_{\sigma} ) &= F^{\prime} \circ S^{{\rm WHA}}( L_{ab} + I_{\sigma} ) \\
&= \sum_{\lambda, \mu \in \Lambda} F^{\prime} \circ S^{{\rm WHA}} \circ \Phi( 1_R \otimes 1_R \otimes \mathbf{e}\genfrac{[}{]}{0pt}{}{(\lambda, a)}{(\mu, b)} + \mathfrak{I}_{\mathbf{w}} ).
\end{align*}
Let us prove that the map \( \tilde{f} := F^{\prime} \circ S^{{\rm WHA}} \circ \Phi \) is the antipode of \( f^+ \). For all \( \alpha \in \mathfrak{A}(w_{\sigma}) \),
\begin{align*}
(\tilde{f} \star f^+)( \alpha )
&= F^{\prime}( S^{{\rm WHA}} ( \Phi( \alpha_{(1)} ) ) \Phi( \alpha_{(2)} ) ) \\
&= F^{\prime}( S^{{\rm WHA}} ( \Phi( \alpha )_{(1)} ) \Phi( \alpha )_{(2)} ) \\
&= F^{\prime}( \varepsilon_s \circ \Phi( \alpha ) ) \\
&= F^{\prime} \circ \Phi \circ \varepsilon_s( \alpha ) ) \\
&= f^+ \circ \varepsilon_s( \alpha )
\end{align*}
The proof for \( f^+ \star \tilde{f} = f^+ \circ \varepsilon_t \) and \( \tilde{f} \star f^+ \star \tilde{f} = \tilde{f} \) is similar. Thus \( \tilde{f} \) is the antipode of \( f^+ \). We can induce that \( \tilde{f} = f^- \) because of the uniqueness of the antipode. Hence \( F^{\prime} = F \) is satisfied.

Finally, we show that \( F \) is a weak bialgebra homomorphism if \( A \) is a weak bialgebra and \( f^{+} \) is a weak bialgebra homomorphism. Let us prove that \( F \) is comultiplicative. Since \( \Delta_{A_{\sigma}} \) and \( \Delta_{A} \) satisfy \eqref{def:wdm}, it suffices to check that \( (F \otimes F) \circ \Delta_{A_{\sigma}}( \alpha + I_{\sigma} ) = \Delta_A \circ F( \alpha + I_{\sigma} ) \). Here, 
\begin{equation*}
\alpha =
\begin{cases}
	g \otimes h \;\; (\forall g, h \in M_{\Lambda}(R)); \\
	L_{ab}; \\
	(L^{-1})_{ab} \;\; (\forall a, b \in X).
\end{cases}
\end{equation*}
If \( \alpha = g \otimes h \; (\forall g, h \in M_{\Lambda}(R)) \), 
\begin{align*}
&(F \otimes F) \circ \Delta_{A_{\sigma}}( g \otimes h + I_{\sigma} ) \\
=& \sum_{\lambda, \mu, \tau \in \Lambda} f^+( g( \lambda ) \otimes e^{(1)} \otimes \mathbf{e}\genfrac{[}{]}{0pt}{}{\lambda}{\tau} + \mathfrak{I}_{\mathbf{w}} ) \otimes  f^+( e^{(2)} \otimes h( \mu ) \otimes \mathbf{e}\genfrac{[}{]}{0pt}{}{\tau}{\mu} + \mathfrak{I}_{\mathbf{w}} ) \\
=& \sum_{\lambda, \mu \in \Lambda} \Delta_A \circ f^+ ( g( \lambda ) \otimes h( \mu ) \otimes \mathbf{e}\genfrac{[}{]}{0pt}{}{\lambda}{\mu} + \mathfrak{I}_{\mathbf{w}} ) \\
=& \Delta_A \circ F( g \otimes h + I_{\sigma} ).
\end{align*}
The proof for \( \alpha = L_{ab} \; (\forall a, b \in X) \) is similar. Let us suppose that \( \alpha = (L^{-1})_{ab} \) for any \( a \) and \( b \in X \). Since \( f^- \colon \mathfrak{A}(w_{\sigma}) \to A^{bop} \) is a weak bialgebra homomorphism, we can induce that
\begin{align*}
&(F \otimes F) \circ \Delta_{A_{\sigma}}( (L^{-1})_{ab} + I_{\sigma} ) \\
=& \sum_{\substack{c \in X \\ \lambda, \mu, \tau, \nu \in \Lambda}} \Delta_{A}( 1_A ) \\
&\times f^-( 1_R \otimes 1_R \otimes \mathbf{e}\genfrac{[}{]}{0pt}{}{(\lambda, c)}{(\mu, b)} + \mathfrak{I}_{\mathbf{w}} ) \otimes  f^-( 1_R \otimes 1_R \otimes \mathbf{e}\genfrac{[}{]}{0pt}{}{(\tau, a)}{(\nu, c)} + \mathfrak{I}_{\mathbf{w}} ) \\
=& \sum_{\substack{c \in X \\ \lambda, \mu, \tau, \nu \in \Lambda}} (f^- \otimes f^-) \circ \Delta_{\mathfrak{A}(w_{\sigma})}^{op}( 1_{\mathfrak{A}(w_{\sigma})} ) \\
& \times f^-( 1_R \otimes 1_R \otimes \mathbf{e}\genfrac{[}{]}{0pt}{}{(\lambda, c)}{(\mu, b)} + \mathfrak{I}_{\mathbf{w}} ) \otimes  f^-( 1_R \otimes 1_R \otimes \mathbf{e}\genfrac{[}{]}{0pt}{}{(\tau, a)}{(\nu, c)} + \mathfrak{I}_{\mathbf{w}} ) \\
=& \sum_{\lambda, \mu, \tau \in \Lambda} f^-( e^{(2)} \otimes 1_R \otimes \mathbf{e}\genfrac{[}{]}{0pt}{}{(\tau, c)}{(\mu, b)} + \mathfrak{I}_{\mathbf{w}} ) \otimes  f^-( 1_R \otimes e^{(1)} \otimes \mathbf{e}\genfrac{[}{]}{0pt}{}{(\lambda, a)}{(\tau, c)} + \mathfrak{I}_{\mathbf{w}} ) \\
=& \sum_{\lambda, \mu \in \Lambda} \Delta_A \circ f^-( 1_R \otimes 1_R \otimes \mathbf{e}\genfrac{[}{]}{0pt}{}{(\lambda, a)}{(\mu, b)} + \mathfrak{I}_{\mathbf{w}} ) \\
=& \Delta_A \circ F( (L^{-1})_{ab} + I_{\sigma} ).
\end{align*}
Hence the map \( F \) is comultiplicative. We prove that \( F \) preseves the counit. Since the counit satisfies \eqref{def:cum}, it suffices to show that 
\begin{align}
\varepsilon_A \circ F( (g \otimes h) L_{ab} + I_{\sigma} ) = \varepsilon_{A_{\sigma}}( (g \otimes h) L_{ab} + I_{\sigma} ); \label{ug+} \\
\varepsilon_A \circ F( (g \otimes h) (L^{-1})_{ab} + I_{\sigma} ) = \varepsilon_{A_{\sigma}}( (g \otimes h) (L^{-1})_{ab} + I_{\sigma} ) \label{ug-}
\end{align}
for any \( g, h \in M_{\Lambda}(R) \), \( a \), and \( b \in X \). For (\ref{ug+}), we can evaluate that
\begin{align*}
\varepsilon_A \circ F( (g \otimes h) L_{ab} + I_{\sigma} ) &= \sum_{\lambda, \mu \in \Lambda} \varepsilon_{\mathfrak{A}(w_{\sigma})}( g( \lambda ) \otimes h( \mu ) \otimes \mathbf{e}\genfrac{[}{]}{0pt}{}{\lambda}{\mu} + \mathfrak{I}_{\mathbf{w}} ) \\
&= \sum_{\lambda, \mu \in \Lambda} \delta_{\lambda, \mu} \delta_{a, b} \phi( g( \lambda ) h( \mu ) ) \\
&= \delta_{a, b} \sum_{\lambda \in \Lambda} \phi( (gh)( \lambda ) ) \\
&= \varepsilon_{A_{\sigma}}( (g \otimes h) L_{ab} + I_{\sigma} ).
\end{align*}
We can use the similar way to prove \eqref{ug-}. This completes the proof.
\end{proof}

\begin{proof}[{\bf Proof of Lemma \ref{lem:estA}}]
We give the proof only for \eqref{lem:es}. For any \( r \in R \) and \( \lambda \in \Lambda \), 
\begin{align*}
\sum_{\mu \in \Lambda} \varepsilon_s( 1_R \otimes r \otimes \mathbf{e}\genfrac{[}{]}{0pt}{}{\mu}{\lambda} + \mathfrak{I}_{\mathbf{w}} )
&= \sum_{\mu, \tau \in \Lambda} \delta_{\lambda, \tau} 1_R \otimes e^{(1)} \psi( e^{(2)} r ) \otimes \mathbf{e}\genfrac{[}{]}{0pt}{}{\mu}{\tau} + \mathfrak{I}_{\mathbf{w}} \\
&= \sum_{\mu \in \Lambda} 1_R \otimes r \otimes \mathbf{e}\genfrac{[}{]}{0pt}{}{\mu}{\lambda} + \mathfrak{I}_{\mathbf{w}}.
\end{align*}
This is the desired conclusion.
\end{proof}

\begin{cor}
If \( \sigma \) is rigid, then the pair \( (A_{\sigma}, \Phi) \) is the Hopf closure of \( \mathfrak{A}(w_{\sigma}) \).
\end{cor}

\begin{ack*}
The auther is deeply grateful to Professor Youichi Shibukawa for his helpful advice.
\end{ack*}


\end{document}